\documentclass[12pt]{article}
\usepackage{amsfonts}
\usepackage{amsmath}
\usepackage{graphicx}
\usepackage{float}
\newtheorem{theorem}{Theorem}

\newtheorem{algorithm}[theorem]{Algorithm}

\newtheorem{conjecture}[theorem]{Conjecture}
\newtheorem{corollary}[theorem]{Corollary}

\newtheorem{lemma}[theorem]{Lemma}
\newtheorem{notation}[theorem]{Notation}

\newtheorem{remark}[theorem]{Remark}

\newenvironment{proof}[1][Proof]{\textbf{#1.} }{\ \rule{0.5em}{0.5em}}
\newdimen\dummy
\dummy=\oddsidemargin 
\bibliography{reference_list}

\begin{document}
\title{Computability of the packing measure of totally disconnected
self-similar sets}
\author{Marta Llorente \\
%EndAName
Dpt. An\'{a}lisis Econ\'{o}mico: Econom\'{\i}a cuantitativa,\\
Universidad Aut\'{o}noma de Madrid, \\
Campus de Cantoblanco, 28049 Madrid \\
%\email{m.llorente@uam.es }
e-mail: m.llorente@uam.es
 \and Manuel Mor\'{a}n \\
%EndAName
Dpt. An\'{a}lisis Econ\'{o}mico I, \\
Universidad Complutense de Madrid, \\
Campus de Somosaguas, 28223 Madrid\\
e-mail: mmoranca@ccee.ucm.es}
\maketitle
\date{}
\begin{abstract}
We present an algorithm to compute the exact value of the packing
measure of self-similar sets satisfying the so called SSC and
prove its convergence to the value of the packing measure. We also
test the algorithm with examples that show both, the accuracy of
the algorithm for the most regular cases and the possibility of
using the additional information provided by it to obtain formulas
for the packing measure of certain self-similar sets. For example,
we are able to obtain a formula for the packing measure of any
Sierpinski gasket with contractio factor in the interval $(0,1/3]$
(Theorem~\ref{packsier}).
\end{abstract}

\section{Introduction and definitions}
In this paper we deal with the problem of computing the value of
the packing measure of (totally disconnected) self-similar sets,
i.e., compact sets $E$ $ \subset \mathbb{R}^{n}$  that satisfy
$E=\bigcup_{i=1}^{N}f_{i}(E)$ for some system $\Psi
=\{f_{1,}f_{2,...,}f_{N}\}$ of contracting similitudes on
$\mathbb{R}^{n}$. The system $\Psi $ is said to satisfy the
\emph{\ open set condition} (\emph{OSC}) if there exists an open,
non-empty, bounded subset $\mathcal{O}\subset \mathbb{R}^{n}$ such
that
\begin{equation}
\bigcup_{i}f_{i}\mathcal{O}\subseteq \mathcal{O}\text{\qquad and
\qquad } f_{i}\mathcal{O}\cap f_{j}\mathcal{O}=\emptyset \quad
\forall i\neq j. \label{osc}
\end{equation}
From now on we shall call \emph{feasible open set }of the system
$\Psi $ (or of $E$) to any set $\mathcal{O}\subset \mathbb{R}^{n}$
satisfying (\ref{osc}). The self-similar sets with separation
conditions are probably the nowadays better understood fractal
sets. For example, it is well known that the \emph{similarity
dimension} of $E$, i.e., the unique solution $s$ of
$\sum_{i=1}^{N}r_{i}^{s}=1$, coincides with the most classical
concepts of metric dimension such as the Hausdorff, packing or
Minkowski dimension when the OSC is fulfilled. Associated to these
dimensions we have the corresponding measures such as the
Hausdorff, centered Hausdorff or packing measures ($H^{s},C^{s} $
and $P^{s}$, respectively). These metric measures are the
appropriate tool to study the size of zero Lebesgue measure sets
in $\mathbb{R}^{n}$, but in most cases they are hard to compute or
estimate computationally. For example, under the OSC, the set $E$
is easily seen to be an $s$-set, i.e., $0<H^{s}(E)<\infty$  (the
same inequality holds for $C^{s}$ and $P^{s}$), but the problem of
finding the precise value of any of these measures even for self
similar sets remains as a challenging open problem. Many efforts
has been done in this direction and the exact values or
estimations for the lower and upper bounds of measures are known
for some fractal sets (see \cite{[Ay]}-\cite{[GZ]},
\cite{[LLM1]}-\cite{[LLM3]}, \cite{[MO],[TTC],[Zh]} and the
references therein). Particularly, in
\cite{[B],[BZ],[BZZL],[Fen0],[GZ]} and \cite{[TTC]}, the authors
use the relation between the packing measure and the lower density
to obtain formulas for the packing measure of some totally
disconnected (but not necessarily self-similar) fractal sets.
Namely, in these papers it holds that
$$
P^{s}(A)= \Big[ \liminf_{d \to 0}
\frac{\mu(B(x,d))}{(2d)^{s}}\Big]^{-1} \textrm{ for } \mu
\text{-a.e. } x \in A \text{,}
$$
where $\mu$ is the natural uniformly distributed probability
measure defined on $A\subset \mathbb{R}^{n}$. Hence, in the
mentioned papers, the problem of computing $P^{s}(A)$ is reduced
to the problem of evaluating the lower density of $\mu$. We
propose to tackle the problem of computing $P^{s}(E)$ from a
different point of view. We continue here the development of the
program on computability of metric measures on self-similar sets,
whose foundations were laid in \cite{[MO]}. Following the lines
developed in \cite{[LLM3]}, where the same problem was considered
for the centered Hausdorff measure, we are going to build an
algorithm able to find the precise or approximate value of the
packing measure of \textit{totally disconnected self-similar sets
}(see (\ref{SSC})). To this aim, the above lower density approach
\ is not suitable as it involves measuring balls of arbitrarily
small radii. However, we know by \cite{[LLM1]} and \cite{[TC]}
that, in the totally disconnected case, it is not necessary to
pass to the limit. In \cite{[MO]} it was proved that
\begin{equation}
P^{s}(E)=\sup \left\{ \frac{(2d)^{s}}{\mu (B(x,d))}:x\in E,B(x,d)\subset
\mathcal{O}\right\} .  \label{moranpack}
\end{equation}%
Using this fact,\emph{\ }it is shown in \cite[Theorem 3.3]{[LLM1]}
that
\begin{equation}
P^{s}(E)=\max \left\{ h(x,d):x\in E,\text{ }\frac{r_{\min
}}{2}\leq d\leq |E| \text{ and }B(x,d)\subset \mathcal{O}\right\}
\text{ }  \label{pssc}
\end{equation}%
(see Remark~\ref{corretion}). Independently, Tricot \cite[Theorem
10.1]{[TC]} proved a version of (\ref{pssc}) where the condition
$B(x,d)\subset \mathcal{O}$ is not needed (see (\ref{tripackme})).
Nevertheless, for numerical purposes, small balls yield problems
such as rounding errors. Moreover, the present algorithm computes
the value\ of $P^{s}(E)$ using \ approximations to the density
functions $h(x,d)$ and it happens that the smaller the balls are
the bigger the error in the numerical approximation of $h(x,d)$
is. Thus, we need to refine Tricot\'s formula so the balls to be
explored are as large as possible. This is the content of our
first theorem.
\begin{theorem}
\label{keylema}Assume that the system $\Psi =\{f_{1,}f_{2,...,}f_{N}\}$ of
contracting similitudes on $\mathcal{\mathbb{R}}^{n}$ satisfies the SSC.
Then, for any $a\in (0,\frac{c}{r_{\min }}]$%
\begin{eqnarray}
P^{s}(E) &=&\max \left\{ h(x,d):x\in E,d\leq a\right\} =  \label{packssc} \\
&=&\max \left\{ h(x,d):x\in E,ar_{\min }\leq d\leq a\right\}
\label{packssc2}
\end{eqnarray}%
\bigskip \nolinebreak (see (\ref{SSC}), (\ref{c}) and (\ref{rminmax}) for
notation).
\end{theorem}
We will prove Theorem~\ref{keylema} in Section~\ref{sec1}.
Section~\ref{description} is devoted to the construction of the
packing measure algorithm built upon (\ref{packssc2}) and in
Theorem~\ref{conv} we manage to prove its convergence to
$P^{s}(E)$. The structure of the algorithm is based on the
algorithm for the centered Hausdorff measure given in
\cite{[LLM3]}, however the extension of previous results is
considerably more involved. The underlying reason is that, while
in the centered Hausdorff measure case we could restrict the
search of optimal balls to balls intersecting at least two
different basic cylinder sets, the competing balls for the packing
measure have radii in a certain closed interval and the nature of
such an interval impedes the restriction to balls touching two
different basic cylinder sets (see Section~\ref{convergsec} for a
detailed discussion). In order to prove Theorem~\ref{conv}, we
need some results from \cite{[LLM3]} and some new lemmas which are
proved at the beginning of the section. Finally, in
Section~\ref{examples}, we test the efficiency of the algorithm as
a tool to give the precise value of the packing measure when the
contractio factors of the similitudes in $\Psi $ are small enough.
In this section we explain how the additional information provided
by the algorithm (the so-called\textbf{\ }candidates for optimal
balls) can be used to rigorously prove explicit formulae for the
exact value of certain self similar sets. For illustration we
collect here the case of the Sierpinski gasket.
\bigskip Let $S_{r}$ be the \emph{self-similar set} associated to the system
$\Psi =\{f_{1,}f_{2,}f_{3}\}$ where
\begin{eqnarray}
f_{1}(\vec{x}) &=&r\vec{x}  \label{sierclas} \\
f_{2}(\vec{x}) &=&r\vec{x}+(1-r,0)  \notag \\
f_{3}(\vec{x}) &=&r\vec{x}+(1-r)(\frac{1}{2},\frac{\sqrt{3}}{2}),  \notag
\end{eqnarray}%
$r\in (0,1)$ and $\vec{x}=(x,y)\in \mathcal{\mathbb{R}}^{2}$. If
$r\in (0,\frac{1}{2})$, then $S_{r}$ is a Sierpinski gasket
satisfying the SSC. We shall denote by $s(r)=\frac{-\log 3}{\log
r}$ to the similarity dimension of the set $S_{r}$. Our methods
prove that,
\begin{theorem}
\label{packsier}If $r\in (0,\frac{1}{3}]$, then
\begin{equation}
P^{s(r)}(S_{r})=\left( 2\frac{1-r}{r}\right) ^{s(r)}.  \label{sierfor}
\end{equation}
\end{theorem}
Theorem~\ref{packsier} extends the formula given by Taylor and Tricot in
\cite{[TTC]} for the case where all the contraction factors are equal to $%
\frac{1}{3}$. As an illustration we indicate at the end of this
section how the algorithm together with the method used in the
proof of (\ref{sierfor}) enable us to recover the known formulas
for the value of the packing measure with an alternative proof. We
also discuss the cases where the contractio ratios are not small
enough to get precise values. Next, we list the definitions and
notation used throughout the paper. Given the system $\Psi
=\{f_{1},...,f_{N}\}$\ of contracting similitudes on $\mathbb{R}^{n}$,
 we shall denote by $r_{i}\in (0,1)$ the \emph{similarity ratio} of $f_{i}\in \Psi $ and write
\begin{equation}
r_{\min }:=\min_{i=1,\ldots ,N}r_{i\text{ }}\qquad \text{and}\qquad r_{\max
}:=\max_{i=1,\ldots ,N}r_{i}\text{.}  \label{rminmax}
\end{equation}
The self-similar set $E$\ (associated to $\Psi $) is\emph{\ totally
disconnected} if
\begin{equation}
f_{i}(E)\cap f_{j}(E)=\emptyset \quad \forall i\neq j,i,j\in \{1,...,N\},
\label{SSC}
\end{equation}
this condition is known as \emph{Strong Separation Condition}
(SSC). We shall assume all the time SSC on the system $\Psi $ and
write
\begin{equation}
c:=\min_{i,j\in \{1,...,N\}}d_{\inf }(f_{i}(E),f_{j}(E))>0,  \label{c}
\end{equation}
where $d_{\inf }(f_{i}(E),f_{j}(E))$ is the distance that
separates $f_{i}(E) $ and $f_{j}(E)$. Regarding the code space we
shall keep the following notation. Let $M:=\{1,...,N\}$ and
\begin{equation*}
M^{k}=\{\mathbf{i}_{k}=(i_{1},...,i_{k}):i_{j}\in M\quad \forall j=1,...,N\}.
\end{equation*}
Given $\mathbf{i}_{k}=i_{1}i_{2}..i_{k}\in M^{k}$, we shall write
$f_{\mathbf{i}_{k}}$ for the similitude
$f_{\mathbf{i}_{k}}=f_{i_{1}}\circ f_{i_{2}}\circ ...\circ
f_{i_{k}}$ with similarity ratio
$r_{\mathbf{i}_{k}}=r_{i_{1}}r_{i_{2}}...r_{i_{k}}$ and given
$A\subset \mathbb{R}^{n}$, we shall denote by
$A_{\mathbf{i}_{k}}=f_{\mathbf{i}_{k}}(A)$ and refer to the sets
$E_{\mathbf{i}_{k}}=f_{\mathbf{i}_{k}}(E)$ as the \emph{cylinder
sets of generation }$k$. The self-similar set $E$ can be written
as the image of the \emph{space of codes} $\mathbb{M}:=M^{\infty
}=M\times M\times ...$ under the\emph{\ projection mapping}\ $\pi
:\mathbb{M}\rightarrow E$ given by
\begin{equation}
\pi (\mathbf{i})=\cap _{k=1}^{\infty }f_{\mathbf{i}(k)}(E)=\cap
_{k=1}^{\infty }E_{\mathbf{i}(k)}  \label{proj}
\end{equation}%
where $\mathbf{i}(k)$ denotes the curtailment $i_{1}...i_{k}\in
M^{k}$ of $\mathbf{i}$ and $f_{\mathbf{i}(k)}=f_{i_{1}}\circ
f_{i_{2}}\circ ...\circ f_{i_{k}}$ We shall denote by $\mu $ the
\textit{natural probability measure}, or \textit{normalized
Hausdorff} \textit{measure}, defined on the ring of cylinder sets
by
\begin{equation}
\mu (E_{\mathbf{i}})=r_{\mathbf{i}}^{s},  \label{mucil}
\end{equation}
and then extended to Borel subsets of $E$. Given $A\subset
\mathbb{R}^{n}$, we shall write $|A|$ for the diameter of $A$ and
for any $\delta \in \mathbb{R}^{+}$, $A_{\delta }=\{x\in
\mathbb{R}^{n}:dist(x,y)\leq \delta \}$ will be the $\delta
-$\textit{parallel neighborhood} of $A$, where $dist(\cdot ,\cdot
)$ denotes the Euclidean distance. The closed ball centered at $x$
and with radius $r>0$ will be denoted by $B(x,d)$ and for the open
ball we shall write $B^{\prime }(x,d)=\left\{ y\in
\mathbb{R}^{n}:dist(x,y)<d\right\} $. Throughout the paper we
shall assume without lost of generality that $R:=|E|=1$.
\section{The packing measure of self-similar sets satisfying the SSC.\label{sec1}}
The \emph{packing measures }were introduced by Tricot
\cite{[T0],[T]} , Taylor and Tricot \cite{[TT],[TTC]} and Sullivan
\cite{[S]}, as the natural metric measure to analyze Brownian
paths and limit sets of Kleinian groups. They are defined by a
two-stage definition using efficient packings: first the
\emph{packing premeasure} is defined by
\begin{equation}
P_{0}^{s}(A)=\lim_{\delta \rightarrow 0}P_{\delta }^{s}(A)  \label{preme}
\end{equation}%
where
\begin{equation}
P_{\delta }^{s}(A)=\sup \left\{ \sum_{i=1}^{\infty }\left\vert
B_{i}\right\vert ^{s}:\left\vert B_{i}\right\vert \leq \delta
,i=1,2,3...\right\}  \label{deltapack}
\end{equation}%
is a non-decreasing \ set function with respect to $\delta $ and the
supremum is taken over all $\delta -$\emph{packings }of $A,$ i.e., countable
collections of disjointed Euclidean balls centered at $A$ and with diameter
smaller than $\delta .$ The \emph{packing measure} is then given by
\begin{equation}
P^{s}(A)=\inf \left\{ \sum_{i=1}^{\infty }P_{0}^{s}(U_{i}):A\subset
\bigcup_{i=1}^{\infty }U_{i}\right\} .  \label{packing}
\end{equation}%
However, this second step (\ref{packing}) may be omitted if the
measured set is (as in our case) a compact set with finite packing
premeasure (see \cite{[Fen]}). Theorem~\ref{keylema}, based on
(\ref{moranpack}), gives an alternative characterization of the
packing measure for self-similar sets satisfying the SSC more
suitable to the computability problem. The main advantage of
working with self-similar sets satisfying the SSC is that we can
guarantee that the supremum in (\ref{moranpack}) is attained (see
\cite[Theorem 3.3]{[LLM1]} and \cite[Theorem 10.1]{[TC]}).
\begin{remark}
\label{corretion}We want to clarify that the statement of \cite[Theorem 3.3]%
{[LLM1]} contains a typo: $B(x,d)\subset \mathcal{O}$ is missing in the
formula of the packing measure. The precise statement which was proved in
\cite{[LLM1]} is
\begin{equation*}
P^{s}(E)=\sup \left\{ \frac{(2d)^{s}}{\mu (B(x,d))}:x\in E,\text{ }\frac{%
r_{\min }}{2}\leq d\leq |E|\text{ and }B(x,d)\subset \mathcal{O}\right\} .
\end{equation*}
\end{remark}
A further advantage of Theorem~\ref{keylema} is that we are able to get rid
of the condition $B(x,d)\subset \mathcal{O}$ at the same time that we
constrain the set of balls where the supremum is to be obtained to balls
having radii on a closed interval bounded away from zero. These results are
possible due to the invariance of the \textit{density function}
\begin{equation*}
\ h(x,d):=\frac{(2d)^{s}}{\mu (B(x,d))}
\end{equation*}
under certain inverse images of the similarity functions of the
system $\Psi$. We recall this fact widely used throughout the
paper.
\begin{lemma}
\label{bolatras}Let $(x,d)\in E\times \mathbb{R}^{+}$ and
$\mathbf{i}\in\mathbb{M}$ with $\pi (\mathbf{i})=x$.
\begin{description}
\item[i)] If $B^{\prime }(x,d)\cap E\subset E_{\mathbf{i}(k)}$,
then
\begin{equation}
h(x,d)=h(f_{\mathbf{i}(k)}^{-1}(x),\frac{d}{r_{\mathbf{i}(k)}})\text{.}
\label{replace1}
\end{equation}
\item[ii)] Assume that for some $k\in \mathbb{N} $, $d\leq
cr_{\mathbf{i}(k)}$ holds, with $r_{\mathbf{i}(k)}=1$ if $k=0$.
Then
\begin{equation}
h(x,d)=h(f_{\mathbf{i}(k+1)}^{-1}(x),\frac{d}{r_{\mathbf{i}(k+1)}}).
\label{replace 3}
\end{equation}
\item[iii)] Let $i\in M$ be such that $r_{i}=r_{\min }$ and
suppose that $d\leq \frac{c}{r_{\min }}$. Then
\begin{equation}
h(x,d)=h(f_{i}(x),dr_{i})  \label{replace 2}
\end{equation}
\end{description}
\end{lemma}
\begin{proof}
 In the situation of \textit{i)} we may write
\begin{equation*}
f_{\mathbf{i}(k)}^{-1}(B^{\prime }(x,d)\cap E)
=f_{\mathbf{i}(k)}^{-1}(B^{\prime }(x,d)\cap
E_{\mathbf{i}(k)})=B^{\prime
}(f_{\mathbf{i}(k)}^{-1}(x),\frac{d}{r_{\mathbf{i}(k)}})\cap E.
\end{equation*}
Then, (\ref{replace1}) holds because, by \cite{[MT]}, we know that
the boundary of any given ball is a $\mu $-null set (see
Remark~3.2 in \cite{[LLM1]}). Assume now that $d\leq
cr_{\mathbf{i}(k)}$. If
\begin{equation}
B^{\prime }(x,d)\cap E\subset E_{\mathbf{i}(k+1)}\text{,}
\label{ik1inclusion}
\end{equation}
then (\ref{replace 3}) holds trivially from \textit{i)}. So we
need only to show that (\ref{ik1inclusion}) holds if $d\leq
cr_{\mathbf{i}(k)}$. Assume, on the contrary, that there exists
$y\in E\cap B^{\prime }(x,d)\setminus E_{\mathbf{i}(k+1)}$ and let
$q=\max \left\{ l:\mathbf{j}(l)= \mathbf{i}(l)\right\} $ where
$\mathbf{j\in }\mathbb{M}$ is such that $\pi (\mathbf{j})=y$.
Then, $q\leq k$, $f_{\mathbf{j}\left( q\right) }^{-1}(y)\in
E_{\mathbf{j}(q+1)}$, $f_{\mathbf{i}\left( q\right) }^{-1}(x)\in
E_{\mathbf{i}(q+1)}$ with $\mathbf{j(}q+1)\neq \mathbf{i(}q+1)$
and
\begin{equation*}
dist(f_{\mathbf{i}\left( q\right) }^{-1}(x),f_{\mathbf{j}\left( q\right)
}^{-1}(y))=r_{\mathbf{i}\left( q\right) }^{-1}dist(x,y)
\end{equation*}
with $f_{\mathbf{i}\left( q\right) }^{-1}=f_{\mathbf{j}\left( q\right)
}^{-1} $ and $r_{\mathbf{i}(q)}^{-1}=1$ if $q=0$. Therefore,
\begin{equation*}
d>dist(x,y)=dist(f_{\mathbf{j(}q)}^{-1}(y),f_{\mathbf{i}\left( q\right)
}^{-1}(x))r_{\mathbf{i}\left( q\right) }\geq cr_{\mathbf{i}\left( q\right)
}\geq cr_{\mathbf{i}(k)},
\end{equation*}
giving the desired contradiction. This shows
\textit{(\ref{ik1inclusion}) } and concludes the proof of
\textit{ii).} Lastly, if $(x,d)\in E\times \mathbb{R}^{+}$ with
$d\leq \frac{c}{r_{\min }}$ and $r_{i}=r_{\min }$, then $
B^{\prime }(f_{i}(x),dr_{i})\subset (E_{i})_{c}$ and
$(E_{i})_{c}\cap E=E_{i} $, so we may write
\begin{eqnarray*}
f_{i}(B^{\prime }(x,d)\cap E) &=&B^{\prime }(f_{i}(x),dr_{i})\cap E_{i}= \\
&=&B^{\prime }(f_{i}(x),dr_{i})\cap (E_{i})_{c}\cap E=B^{\prime
}(f_{i}(x),dr_{i})\cap E
\end{eqnarray*}
and \textit{iii)} follows.
\end{proof}

Now we turn to the proof of Theorem~\ref{keylema} whose aim is a
reduction of the set of balls where the supremum in
(\ref{moranpack}) is to be computed. From a computational point of
view this reduction is more efficient if the balls to be explored
are larger, so the idea is to seek the largest possible balls
which still give the packing measure. In
\cite[Theorem~10.1]{[TC]}, Tricot obtained the following result in
this direction: "If $E$ is totally disconnected then
\begin{equation}
P_{0}^{s}E=P^{s}E=\frac{1}{m}  \label{tripackme}
\end{equation}%
where $m=\inf I_{E}$ with $I_{E}=\{\frac{\mu
(B(x,d))}{(2d)^{s}}:x\in E,$ $\frac{r_{\min }}{r_{\max }}c\leq
d\leq \frac{1}{r_{\max }}c\}$". Theorem~\ref{keylema} is an
extension of Tricot's result more suitable to our purposes and
proved with different arguments.

\begin{proof}[Proof of Theorem~\protect\ref{keylema}]
For $\delta >0$, let $\mathcal{A}(\delta ):=\left\{ (x,d):x\in
E,0<d\leq \delta \right\} $ and let $S(\delta ):=\sup
\left\{h(x,d):(x,d)\in A(\delta )\right\} .$ Consider also
\begin{eqnarray*}
\mathcal{A}(\delta _{1},\delta _{2}) &:&=\left\{ (x,d):x\in E,\delta
_{1}\leq d\leq \delta _{2}\right\} \text{ \ \ and} \\
S(\delta _{1},\delta _{2}) &:&=\sup \left\{ h(x,d):(x,d)\in
\mathcal{A}(\delta _{1},\delta _{2})\right\}.
\end{eqnarray*}
Let $a,b\in (0,\frac{c}{r_{\min }}]$ and suppose without loss of generality
that\textbf{\ }$b<a$ . We are going to show first that
\begin{equation}
S(a)=S(b).  \label{igualdadS}
\end{equation}
Notice that,\textbf{\ }if $a\in (0,\frac{c}{r_{\min }}]$,\textbf{\
}(\ref{replace 2}) implies
\begin{equation*}
S(a)\leq S(ar_{\min }).
\end{equation*}%
The opposite inequality also holds as
$\mathcal{A}(ar_{\min})\subset \mathcal{A}(a)$. Thus, for any
$a\in (0,\frac{c}{r_{\min }}]$ and $k\in \mathbb{N}^{+}$, we have
\begin{equation}
S(a)=S(ar_{\min })=S(ar_{\min }^{k})\text{.}  \label{igualdadSsmall}
\end{equation}
This shows that, for any $k\in \mathbb{N}^{+}$\ such that
$ar_{\min }^{k}<b$, $S(b)\leq S(a)=S(ar_{\min }^{k})\leq S(b)$
whence
\begin{equation*}
S(b)=S(a),
\end{equation*}
concluding the proof of (\ref{igualdadS}). Now take $(x,d)\in
\mathcal{A}(b)$ and let $\mathbf{i\in }\mathbb{M}$ be such that
$\pi (\mathbf{i)=}x$ and $ k:=\min \left\{ l\in
\mathbb{N}:\frac{d}{r_{\mathbf{i}(l)}}\geq ar_{\min }\right\} $.
Then,
\begin{equation*}
ar_{\min }\leq \frac{d}{r_{\mathbf{i}(k)}}\leq
\frac{d}{r_{\mathbf{i} (k-1)}r_{\min }}<a\leq \frac{c}{r_{\min }},
\end{equation*}%
where the first and the third inequalities hold by the selection
of $k.$ This shows that $\
(f_{\mathbf{i}(k)}^{-1}(x),\frac{d}{r_{\mathbf{i}(k)}})\in \mathcal{A}(ar_{\min },a).$ Moreover, since $d\leq cr_{\mathbf{i}(k-1)}$%
, part\textit{\ ii)} of Lemma~\ref{bolatras} implies that
$h(f_{\mathbf{i}(k)}^{-1}(x),\frac{d}{r_{\mathbf{i}(k)}})=h(x,d)$
and, by (\ref{igualdadS}), we obtain
\begin{equation*}
S(b)\leq S(ar_{\min },a)\leq S(a)=S(b)
\end{equation*}
that is
\begin{equation}
S(a)=S(b)=S(ar_{\min },a)  \label{newigualdad}
\end{equation}
for any $a$, $b\in (0,\frac{c}{r_{\min }}]$.
Lastly we prove that $P^{s}(E)=S(ar_{\min },a)$. To this end, let
$k:=\min \{l\in \mathbb{N}^{+}:2r_{\max }^{k}+cr_{\min }^{k}\leq
c\}$ and take two feasible open sets
for $E$,\ namely $\mathcal{O}_{1}\mathcal{=}(E)_{\frac{c}{2}}$ and $\mathcal{%
O}_{2}\mathcal{=}(E)_{r_{\min }^{k}\frac{c}{2}}$. On one hand,
(\ref{newigualdad}) with $b=\frac{c}{2}$ together with
(\ref{moranpack}) applied to $\mathcal{O}_{1}$ give
\begin{equation}
P^{s}(E)\geq S(\frac{c}{2})=S(ar_{\min },a)\text{.}  \label{ineq}
\end{equation}
On the other hand, the connectivity of the Euclidean balls,\ imply
that any $ B(x,d)\subset \mathcal{O}_{2}$\ with $x\in
E_{\mathbf{i}(k)}$\ must be contained in some $(E_{i})_{r_{\min
}^{k}\frac{c}{2}},i\in M,$\ so it must verify that
\begin{equation*}
d\leq |E_{\mathbf{i}(k)}|+\frac{c}{2}r_{\min }^{k}\leq r_{\max }^{k}+\frac{c%
}{2}r_{\min }^{k}\text{.}
\end{equation*}
Hence, if we apply (\ref{moranpack}) to $\mathcal{O}_{2}$, (\ref{ineq})
implies that
\begin{equation*}
S(ar_{\min },a)\leq P^{s}(E)\leq S(r_{\max
}^{k}+\frac{c}{2}r_{\min}^{k})\leq S(\frac{c}{2})=S(ar_{\min },a)
\end{equation*}
which shows the desired equality. Notice that, since the
$\mu-$measure of boundaries of balls is null, $h$ is a continuous
function in $E\times \mathbb{R}^{n}$ and the supremum in the
definition of $S(ar_{\min },a)$ is attained (the set
$\mathcal{A}(ar_{\min },a)$ is compact), this observation ends the
proof of the theorem.
 \end{proof}

 The following lemma will allow us to narrow the search for
optimal balls to those whose boundary intersects $E$ (see also
Corollary~\ref{newattain}).
\begin{lemma}
\label{attainE}Let $a\in (0,\frac{c}{r_{\min }}]$ and
$(x_{0},d_{0})\in E\times \lbrack ar_{\min },a]$ be such that
\begin{equation}
P^{s}(E)=h(x_{0},d_{0})=\frac{(2d_{0})^{s}}{\mu (B(x_{0},d_{0}))}.
\label{minpair}
\end{equation}
Then, either $d_{0}=a$ or
\begin{equation}
\partial B(x_{0},d_{0})\cap E\neq \emptyset .  \label{nonemptybdd}
\end{equation}
\end{lemma}
\begin{proof}
Let $(x_{0},d_{0})\in E\times \lbrack ar_{\min },a]$ satisfying
(\ref{minpair}) and let $\mathbf{i=}i_{1}i_{2}...\in \mathbb{M}$
be such that $\pi (\mathbf{i})=x_{0}$. Suppose on the contrary
that $\partial B(x_{0},d_{0})\cap E=\emptyset $ and $d_{0}<a$.
Then, there exists $d^{\prime }>d_{0}$ such that $ar_{\min }\leq
d^{\prime }\leq a$ and
\begin{equation*}
(B(x_{0},d^{\prime })\setminus B(x_{0},d_{0}))\cap E=\emptyset .
\end{equation*}
Whence, $\mu (B(x_{0},d_{0}))=\mu (B(x_{0},d^{\prime }))$
contradicting the maximality of $(x_{0},d_{0})$ (see
(\ref{packssc2})).
\end{proof}
\begin{remark}
Observe that the case $d_{0}=a$ in Lemma~\ref{attainE} might be omitted,
this is because (\ref{replace 2}) implies that for any $x\in E$,
\begin{equation*}
h(x,a)=h(f_{i}(x),ar_{\min })
\end{equation*}
where $i\in M$ is such that $r_{\min }=r_{i}$.
\end{remark}
Two straightforward consequences of the above remark are the following
corollaries to Theorem~\ref{keylema} and Lemma~\ref{attainE}, respectively.
\begin{corollary}
\label{maxallgo}Under the conditions of Theorem~\ref{keylema},
\begin{equation*}
P^{s}(E)=\max \{h(x,d):x\in E\text{ \ \ and \ \ }d\in \lbrack
ar_{\min},a)\}.
\end{equation*}
\end{corollary}
\begin{corollary}
\label{newattain}For any $a\in (0,\frac{c}{r_{\min }}]$ there
exists $(x_{0},d_{0})\in E\times \lbrack ar_{\min },a)$ such that
\begin{equation*}
P^{s}(E)=h(x_{0},d_{0})
\end{equation*}
and
\begin{equation*}
\partial B(x_{0},d_{0})\cap E\neq \emptyset.
\end{equation*}
\end{corollary}
It is useful to note that, for any pair $(x_{0},d_{0})\in E \times
\lbrack ar_{\min },a)$ satisfying (\ref{minpair}),
Corollary~\ref{newattain} guarantees the existence of a point
$y\in E$ such that
\begin{equation}
P^{s}(E)=\frac{(2d_{0})^{s}}{\mu
(B(x_{0},d_{0}))}=\frac{(2dist(x_{0},y))^{s}}{
\mu(B(x_{0},dist(x_{0},y))}.  \label{atta}
\end{equation}
\section{\protect\bigskip Description of the algorithm\label{description}}
This section is devoted to describe an algorithm to compute the
packing measure of self-similar sets satisfying the SSC. We recall
that, for this particular class of self-similar sets, the packing
measure can be defined as
\begin{equation}
P^{s}(E)=\max \left\{ h(x,d):x\in E,ar_{\min }\leq d\leq a\right\}
\label{formulapacking}
\end{equation}%
where $a$ is any real number within the interval $(0,\frac{c}{r_{\min }}]$
(see Theorem~\ref{keylema}).
Our method is strongly based on (\ref{formulapacking}) as, to find the value
of $P^{s}(E)$, we construct an algorithm for maximizing the value of
\begin{equation}
h(x,d)=\frac{(2d)^{s}}{\mu (B(x,d))}  \label{densidadcont}
\end{equation}
when $x\in E$\ and $\tilde{c}\leq d\leq \frac{\tilde{c}}{r_{\min
}}$, where $\tilde{c}$ is an estimate of $c$ . The idea is to \
estimate $c$ with $\tilde{c}$ and construct sequences $ \left\{
A_{k}\right\} $\ of finite sets and $\{\mu _{k}\}$ of discrete
measures supported on $A_{k}$ such that $\overline{\cup
_{k=1}^{\infty }A_{k}}=E$ and $\{\mu _{k}\}$ converges weakly to
$\mu $. $\overline{A}$ stands for the closure of $A$. This allows
us to construct another\ sequence $ \left\{ \tilde{m}_{k}\right\}
$ converging to $P^{s}(E)$ by choosing on each step $k$, a pair
$(\tilde{x}_{k},\tilde{y}_{k})$ $\in $ $A_{k}\times A_{k}$
satisfying
\begin{gather}
\tilde{m}_{k}:=h_{k}(\tilde{x}_{k},dist(\tilde{x}_{k},\tilde{y}_{k}))=
\label{mktildedef} \\
=\max \{h_{k}(x,dist(x,y)):(x,y)\in A_{k}\times A_{k}\text{ and }\tilde{c}%
\leq dist(x,y)\leq \frac{\tilde{c}}{r_{\min }}\},  \notag
\end{gather}%
where $h_{k}(x,d):=\frac{(2d)^{s}}{\mu _{k}(B^{\prime }(x,d))}$.
\begin{remark}
Notice that the definition of the discrete density function $h_{k}(x,r)$\
uses open balls rather than\ the closed balls used in the continuous version
$h(x,d)$. Actually, since $\mu (\partial B(x,d))=0$\ (see (\cite{[MT]})), it
is also possible to use open balls in (\ref{densidadcont}). The situation
with $h_{k}(x,d)$\ is slightly different. Either the use of open or closed
balls in the definition of $h_{k}(x,d)$\ leads to a convergent algorithm,
but the numerics have proved that for the packing measure is more convenient
to use open balls while, in the centered Hausdorff measure case, closed
balls were more adequate. The difference between these two cases relies on
the nature of the candidates to optimal balls: in the first case they have
to be as emptiest as possible while in the present case, the fuller the
better.
\end{remark}
\subsection{Homogeneous case $(r_{i}=r$ $\forall i\in M)$}
Next, we describe the algorithm for self-similar sets where all the
contraction ratios coincide as this case illustrates better the central idea
of the construction. Afterwards, we shall explain in Section~\ref{mod} the
modifications needed to treat the case of unequal similarity ratios. Observe
that if $r_{i}=r_{j}:=r$ $\forall $ $i\neq j$, the \textit{invariant measure
}$\mu $ satisfies that%
\begin{equation}
\mu (E_{\mathbf{i}_{k}})=r^{ks}=\frac{1}{m^{k}}\text{\quad }\forall \text{ }%
\mathbf{i}_{k}\in M^{k}.  \label{mucilig}
\end{equation}
\begin{algorithm}
\label{case1}(Homogeneous case: $r_{i}=r_{j}:=r$ \ $\forall i\neq
j,$ $ i,j\in M$) \emph{Input of the Algorithm}: System of
contracting similitudes,$\ k_{\max }$\ (the number of iterations),
$\tilde{k}$ and $N\geq 2$ (see step 3).
\begin{enumerate}
\item \textbf{Construction of }$\mathbf{A}_{k}$. Let \ $A_{1}=\left\{
x_{1},x_{2},...,x_{m}\right\} $ be the set of the fixed points for the
similitudes in $\Psi $, that is, for every $i\in M$, $f_{i}(x_{i})=x_{i}$.
For $k\in \mathbb{N}^{+}$, let $A_{k}=S\Psi (A_{k-1})$ be the set of $m^{k}$
points obtained by applying $S\Psi (x)=\bigcup\limits_{i\in M}f_{i}(x)$ to
each of the $m^{k-1}$ points of $A_{k-1}$.
\begin{notation}
\label{not2} For every $x\in A_{k}$ we shall denote by
$\mathbf{i}_{k}(x)=i_{1}(x)....i_{k}(x)\in M^{k}$ the unique
sequence of length $k$ such that $x=f_{\mathbf{i}_{k}(x)}(y)$ for
some $y\in A_{1}$ and we shall write $x$ as
$x_{\mathbf{i}_{k}(x)}$. Note that, in this case,
$y=f_{i_{k}(x)}(y)$ and that $x\in A_{k}\setminus A_{k-1}\iff
i_{k}(x)\neq i_{k-1}(x)$.
\end{notation}
\item \textbf{List of distances.} This step consists on computing
the set
\begin{equation*}
\Delta _{k}:=\{dist(x,y):(x,y)\in A_{k}\times A_{k}\text{ }\}
\end{equation*}
of distances between the pairs of points in $A_{k}\times A_{k}$.
It is important to notice that $\Delta _{k-1}\subset \Delta _{k}$
since $ A_{k-1}\subset A_{k}$ (see Lemma~\ref{keylemaconv}~(ii)).
It is then clear that, in order to construct $\Delta _{k}$, there
is not need to compute again the distances already computed in
$\Delta _{k-1}$, so we calculate only the distances $dist(x,y)$
between those points $(x,y)\in A_{k}\times A_{k}$ satisfying that
\begin{equation*}
i_{k}(x)\neq i_{k-1}(x)\text{ or }i_{k}(y)\neq i_{k-1}(y)
\end{equation*}
(see Notation~\ref{not2}). For every $k\in \mathbb{N}^{+}$, let us
denote by $\Delta _{k}^{0}$ the set of all these distances and set
$\Delta _{k}=\Delta _{k}^{0}\cup \Delta _{k-1}$, with $\Delta
_{1}=\Delta _{1}^{0}=\{dist(x,y):(x,y)\in A_{1}\times A_{1}\}$.
Observe that
\begin{equation*}
\Delta _{k}^{0}:=\{dist(x,y):x,y\in A_{k}\}\setminus \Delta
_{k-1}.
\end{equation*}
From now on we assign the code $(\mathbf{i}_{k}(x),$
$\mathbf{i}_{k}(y))$ to each $dist(x,y)\in \Delta _{k}$ and refer
to $(\mathbf{i}_{k}(x),$ $\mathbf{i}_{k}(y))$ as the\emph{\
}$k-$\emph{address of }$dist(x,y)$. Notice that if $dist(x,y)\in
\Delta _{k}$ , its $(k+1)-$\emph{address} will be
\begin{equation*}
(\mathbf{i}_{k}(x)i_{k},\mathbf{i}_{k}(y)i_{k}(y))=(i_{1}(x)....i_{k}(x)i_{k}(x),i_{1}(y)....i_{k}(y)i_{k}(y)).
\end{equation*}
\item \textbf{Estimation of }$\mathbf{c}$ \textbf{by}
$\mathbf{\tilde{c}}$. Let $c_{0}:=\min \{dist(x,y)\in
\Delta_{1}:x\neq y\}$. For every $k\in  \mathbb{N}^{+}$,
\begin{equation}
\tilde{c}_{k}:=c_{k}-2r^{k}  \label{cktilde0}
\end{equation}%
where%
\begin{equation*}
c_{k}:=\min \{c_{k-1},c_{k}^{\prime }\}
\end{equation*}%
and
\begin{equation*}
c_{k}^{\prime }=\min \{dist(x,y)\in \Delta _{k}^{0}:i_{1}(x)\neq i_{1}(y)\}.
\end{equation*}%
Notice that
\begin{equation*}
\tilde{c}_{k}=\min_{i,j\in M}\{dist(f_{i}(A_{k-1}),f_{j}(A_{k-1})):i\neq
j,\}-2r^{k}.
\end{equation*}%
We define
\begin{equation}
\tilde{c}:=\tilde{c}_{\tilde{k}}  \label{ctilde}
\end{equation}%
where $\tilde{k}$ is the biggest $k\in
%TCIMACRO{\U{2115} }%
%BeginExpansion
\mathbb{N}
%EndExpansion
^{+}$ allowed by the computer capacity such that $\tilde{c}_{k}>0$.
\begin{remark}
In many examples and, in particular, when all the similitudes in $\Psi $ are
homotheties, the above approximation of $c$\ by $\tilde{c}$\ is not needed
because the minimal distance between basic cylinder sets is known. In these
cases, the present step should be replaced with the value of $c$ in the
construction of the algorithm. It is easy to see that Theorem~\ref{conv}
also holds if  $c$\ is replaced with $\tilde{c}$ in the construction of the
sequence $\{\tilde{m}_{k}\}$.
\end{remark}
\item \textbf{Construction of }$\mathbf{\mu }_{k}$. For all $k\in \mathbb{N}%
^{+}$, set
\begin{equation}
\mu _{k}(x)=\frac{1}{m^{k}}\text{\quad }\forall x\in A_{k.}
\label{muigulapto}
\end{equation}%
Thus,
\begin{equation*}
\mu _{k}=\frac{1}{m^{k}}\sum_{i=1}^{m^{k}}\delta _{x_{i}}
\end{equation*}%
is a probability measure with $spt(\mu _{k})=A_{k}=\{x_{1},...,x_{m^{k}}\}$.
\item \textbf{Construction of }$\mathbf{\tilde{m}}_{k}$
Given $x\in A_{k}$ :
\begin{itemize}
\item[5.1] Arrange in increasing order the distances $d\in \Delta _{k}$
containing $\mathbf{i}_{k}(x)$ in their addresses and such that $d\leq \frac{%
\tilde{c}}{r}$.
\item[5.2] Let $L(x):=\{d_{0},d_{1},....,d_{p(x)}\}$\ with $0\leq p(x)\leq
m^{k}-1$, be the list of ordered distances. For each $j=0,...,p(x)$, $%
B^{\prime }(x,d_{j})$\ contains $j-t(j)$\ points of $\ A_{k}$, where $%
t(j)=\#\{l<j:d_{l}=d_{j}\}$, therefore
\begin{equation}
\mu _{k}(B{\acute{}}(x,d_{j}))=\frac{j-t(j)}{m^{k}}.
\label{mukigualball}
\end{equation}
Compute
\begin{equation}
h_{k}(x,d_{j}):=\frac{(2d_{j})^{s}}{\mu
_{k}(B{\acute{}}(x,d_{j}))}=\frac{(2d_{j})^{s}}{\frac{j-t(j)}{m^{k}}}=\frac{m^{k}\left(
2d_{j}\right) ^{s}}{j-t(j)}  \label{fkigual}
\end{equation}
only for those distances $d_{j}\in L(x)$\ satisfying that
\begin{equation*}
\tilde{c}\leq d_{j}\text{.}
\end{equation*}
\item[5.3] Find the maximum, $m_{k}(x)$, of the values computed in the step
5.2.
\item[5.4] Repeat steps 5.1-5.3 for each $x\in A_{k}.$
\item[5.5] Take the maximum,
\begin{equation*}
\tilde{m}_{k}:=\max \{m_{k}(x):x\in A_{k}\}
\end{equation*}
of the $m^{k}$ values computed in step 5.4.
\end{itemize}
\item If $k=k_{\max }$\ end the program. If $k<k_{\max }$\ let
$k=k+1$\ and go to step~1.
\end{enumerate}
\end{algorithm}
We recall again the importance from the computational point of view of
reducing the set of balls where the supremum is to be computed. Moreover,
these balls should be as large as possible. This is the reason to build the
algorithm upon the formula
\begin{equation}
P^{s}(E)=\max \left\{ h(x,d):(x,d)\in E\times \lbrack
c,\frac{c}{r_{\min }} ]\right\} \text{.}  \label{packssc3}
\end{equation}%
Notice that, for every $x\in A_{k}$, the algorithm finds the maximum value
of $h_{k}(x,d_{j})$ only for those pairs $(x,d_{j})\in A_{k}\times L(x)$
such that
\begin{equation}
d_{j}\in \lbrack \tilde{c},\frac{\tilde{c}}{r_{\min
}}]=[\tilde{c},\frac{\tilde{c}}{r}].  \label{condition}
\end{equation}
\subsection{General case $(r_{i}\neq r_{j})$\label{mod}}
Next we list the changes needed to build the algorithm when the contraction
ratios are unequal. The main difference with the previous case is the value
assigned to the measures $\mu _{k}$. The structure of the algorithm is the
same in either case.
\begin{enumerate}
\item In step 3, replace (\ref{cktilde0}) with
\begin{equation}
\tilde{c}_{k}:=c_{k}-2r_{\max }^{k}.  \label{ctildege}
\end{equation}
\item In step 4, replace (\ref{muigulapto}) with
\begin{equation}
\mu _{k}(x)=r_{\mathbf{i}_{k}(x)}^{s}\text{\quad }\forall x\in A_{k},
\label{mukdif}
\end{equation}
(see Notation~\ref{not2}). Consequently, $\mu _{k}$ is a probability measure
with $spt(\mu _{k})=A_{k}=\{x_{1},...,x_{m^{k}}\}$ given by
\begin{equation}
\mu _{k}=\sum_{j=1}^{m^{k}}r_{\mathbf{i}_{k}(x_{j})}^{s}\delta
_{x_{j}} . \label{sumdeltasdif}
\end{equation}
\item In step 5.1 write $d\leq \frac{\tilde{c}}{r_{\min }}$
instead of $ d\leq \frac{\tilde{c}}{r}$. \item In step 5.2,
replace (\ref{mukigualball}) with
\begin{equation}
\mu _{k}(B^{\prime
}(x,d_{j}))=\sum_{q=0}^{j-t(j)-1}r_{\mathbf{i}_{k}(x_{i_{q}})}^{s},
\label{mudifball}
\end{equation}
where $x_{i_{q}}\in A_{k}$\ is such that
$d_{q}=dist(x,x_{i_{q}})<d_{j}$\ for all $q=0,...,j-t(j)-1$.
Observe that $\mu _{k}(B^{\prime }(x,d_{0}))=0$\ as $x=x_{i_{0}}$
in every case. Replace also (\ref{fkigual}) with
\begin{equation}
h_{k}(x,d_{j}):=\frac{(2d_{j})^{s}}{\mu _{k}(B^{\prime
}(x,d_{j}))}=\frac{(2d_{j})^{s}}{\sum_{q=0}^{j-t(j)-1}r_{\mathbf{i}_{k}(x_{i_{q}})}^{s}}.
\label{fkdif}
\end{equation}
Observe that the last equality in (\ref{condition}) does not hold in this
general case.
\end{enumerate}
\begin{remark}
\label{relations}Since $|E|=1$ and for every $k\in
\mathbb{N}^{+}$, $E\subset (A_{k})_{r_{\max
}^{k}}=\{x\in\mathbb{R}^{n}:dist(x,A_{k})\leq r_{\max }^{k}\}$, we
have that
\begin{equation*}
c_{k}-2r_{\max }^{k}\leq c
\end{equation*}
and thus
\begin{equation}
c-2r_{\max }^{\tilde{k}}\leq \tilde{c}\leq c  \label{ckbound}
\end{equation}
(see (\ref{ctilde}) for notation). Hence, by
(\ref{formulapacking})
\begin{equation}
P^{s}(E)=\max \{h(x,d):(x,d)\in E\times \lbrack
\tilde{c},\frac{\tilde{c}}{ r_{\min }}]\}.  \label{maxalg}
\end{equation}
\end{remark}
\begin{notation}
For the rest of the paper we shall keep the following notation.
Let $ A_{k}=S\Psi (A_{k-1})$ be the set of $m^{k}$\ points
obtained after $k$ iterations with $A_{1}=\{x\in
\mathbb{R}^{n}:f_{i}(x)=x,i=1,...,m\}$, we write
\begin{equation*}
A:=\cup _{k=1}^{\infty }A_{k}.
\end{equation*}
Given $k\in \mathbb{N}^{+}$ and $x\in A_{k}$, let $D_{k}^{x}$ be the set of
distances satisfying condition (\ref{condition}), we denote by
\begin{equation*}
D_{k}:=\cup _{x\in A_{k}}D_{k}^{x}
\end{equation*}%
and we write
\begin{equation*}
D:=\cup _{k=0}^{\infty }D_{k}.
\end{equation*}
Observe that (\ref{condition}) implies that $D_{k}$ and $D$ only
take values in the interval $[\tilde{c},\frac{\tilde{c}}{r_{\min
}}]$. Under this notation the sequence $\tilde{m}_{k}$ computed in
step~5.5 can be written as
\begin{equation*}
\tilde{m}_{k}:=h_{k}(\tilde{x}_{k},\tilde{d}_{k}):=\frac{(2\tilde{d}_{k})^{s}}{\mu
_{k}(B^{\prime }(\tilde{x}_{k},\tilde{d}_{k}))}=\max
\{h_{k}(x,d):(x,d)\in A_{k}\times D_{k}\}.
\end{equation*}
We shall refer to the sequence
$\{(\tilde{x}_{k},\tilde{d}_{k})\}_{k=1}^{\infty }$ as the
(optimal) \emph{algorithm sequence}.
\end{notation}
\section{Convergence of the algorithm\label{convergsec}}
In this section \ we show the convergence of\ the \emph{algorithm
sequence} $\{\tilde{m}_{k}\}$. This is stated in the following
theorem.
\begin{theorem}
\label{conv} The \emph{algorithm sequence
}$\{\tilde{m}_{k}\}_{k\in \mathbb{N}^{+}}$ given by (\ref{seal})
converges to $P^{s}(E)$. Moreover, for any $k\in \mathbb{N}^{+}$
such that $\ \mu _{k}(B^{\prime
}(\widetilde{x}_{k},\widetilde{d}_{k}))=\mu
(B(\widetilde{x}_{k},\widetilde{d}_{k}))$,
\begin{equation}
P^{s}(E)\geq \tilde{m}_{k}.  \label{lowerbound}
\end{equation}
\end{theorem}
The proof of Theorem~\ref{conv} is postponed to the end of the
section. It will follow the structure of the proof given in
\cite{[LLM2]}, where equivalent results were obtained for the
centered Hausdorff measure. However, the case of the packing
measure is structurally more difficult. On one hand, the value of
the centered Hausdorff measure was found as the maximum of the
densities of balls intersecting at least two different basic
cylinder sets. This is not the case of the packing measure since
the balls upon which we seek the minimum value have radii bounded
between $\tilde{c}$ and $\frac{\tilde{c}}{r_{\min
}}<\frac{c}{r_{\min }}$. This means that we cannot restrict to
balls touching two different basic cylinder sets since the radii
these balls might be bigger than $\frac{c}{r_{\min }}$ if some of
the contractio factors of the similarities are bigger than $c$. A
preliminary step will be to show the existence of a sequence
$\{(x_{k},d_{k})\}_{k=1}^{\infty }$ in $A_{k}\times $
$[\tilde{c},\frac{ \tilde{c}}{r_{\min }}]$ such that
\begin{equation*}
m_{k}:=h_{k}(x_{k},d_{k}):=\frac{(2d_{k})^{s}}{\mu _{k}(B^{\prime
}(x_{k},d_{k}))}\rightarrow P^{s}(E).
\end{equation*}
This is done in Lemma~\ref{gooddk}. It is important to notice
that, although Lemma~\ref{gooddk} \ guarantees the existence of a
nice sequence $\left\{ m_{k}\right\} $ converging to $P^{s}(E)$,
there is a priory no reason for this sequence to coincide with the
\emph{algorithm sequence} $\{\tilde{m}_{k}\}$. The algorithm
selects its own sequence $\{\tilde{m}_{k}\}_{k=1}^{\infty }$ by
choosing on each step $k$, those pairs $(\tilde{x}_{k},\tilde{d}_{k})$ $\in $ $A_{k}\times D_{k}$ $\subset A_{k}\times $ $[%
\tilde{c},\frac{\tilde{c}}{r_{\min }}]$ satisfying
\begin{equation}
\tilde{m}_{k}:=h_{k}(\tilde{x}_{k},\tilde{d}_{k}):=\frac{(2\tilde{d}_{k})^{s}
}{\mu _{k}(B^{\prime }(\tilde{x}_{k},\tilde{d}_{k}))}=\max
\{h_{k}(x,d):(x,d)\in A_{k}\times D_{k}\},  \label{seal}
\end{equation}%
and even if $\left\{ m_{k}\right\} $ and $\{\tilde{m}_{k}\}$
coincide, $B^{\prime }(x_{k},d_{k})$ and $B^{\prime
}(\tilde{x}_{k},\tilde{d}_{k})$ might be different. Actually, for
each $k\in \mathbb{N}^{+}$ there could be more than one pair
$(\tilde{x}_{k},\tilde{d}_{k})$ satisfying (\ref{seal}). However,
an important feature of the algorithm is that, in many cases, it
gives the candidate to optimal ball. At the end of the section we
shall prove that the sequence of maximal values
$\{\tilde{m}_{k}\}$ converges to the maximum $P^{s}(E)$ (see
Theorem \ref{conv}). We shall need the following result from
\cite{[LLM2]} to show the existence of the sequences given in
Lemma~\ref{gooddk}.
\begin{theorem}[\cite{[LLM2]} Theorem 4.9]
\label{main}Let $x_{0},y_{0}\in E$ \ with $x_{0}\neq y_{0}$ and\
let $d_{0}=dist(x_{0},y_{0})$. If the sequences $\{x_{k}\}$,
$\{y_{k}\}$ are such that $x_{k},y_{k}\in A_{k}$ $\forall k\in
\mathbb{N}^{+}$, $x_{k}\rightarrow x_{0}$ and $y_{k}\rightarrow
y_{0}$, then $d_{k}:=dist(x_{k},y_{k}) \rightarrow d_{0}$ and
\begin{equation}
m_{k}:=\frac{(2d_{k})^{s}}{\mu _{k}(B^{\prime }(x_{k},d_{k}))}\rightarrow
\frac{(2d_{0})^{s}}{\mu (B(x_{0},d_{0}))}.  \label{convergence}
\end{equation}
\end{theorem}
\begin{remark}
The formulation of Theorem \ref{main} in \cite{[LLM2]} is slightly
different. Namely $d:=dist(x_{0},y_{0})\in \lbrack c,1]$ and the
definition of $m_{k}$\ uses closed instead of open balls. However,
the proof only uses the fact that $d$ is bounded away from zero
and infinity, (which holds for $d_{0}$) and it is easy to see that
it also works with open balls. Notice that the boundary of any
ball is a $\mu -$null set. Therefore \cite[Theorem 4.9]{[LLM2]}
can be reformulated as above.
\end{remark}
Before stating Lemma~\ref{gooddk}, let us recall some basic
results given in \cite{[LLM2]} that we shall use to show the
convergence of the \emph{algorithm sequence}.
\begin{lemma}[\cite{[LLM2]} Lemma 4.1]
\label{keylemaconv}
\begin{enumerate}
\item[(i)] For every $x\in E$ there exists a sequence $\{x_{k}\}$
with $ x_{k}\in A_{k}$ such that $\lim_{k\rightarrow \infty
}x_{k}=x$. \item[(ii)] For every $k\in \mathbb{N}^{+}$,
\begin{equation*}
A_{k}\subset A_{k+1}.
\end{equation*}
\item[(iii)] Let $k\in \mathbb{N}^{+}$, $x\in A_{k}$ and $\mathbf{i}_{k}\in
M^{k}$ be such that $\ f_{\mathbf{i}_{k}}(y)=x$ for some $y\in A_{1}$, then
\begin{equation}
\mu _{k}(x)=\mu _{k}(E_{\mathbf{i}_{k}})=\mu (E_{\mathbf{i}_{k}}).
\label{keyinq}
\end{equation}
\item[(iv)] The sequence $\{\mu _{k}\}_{k\in \mathbb{N}^{+}}$ converges
weakly to $\mu $ and thus
\begin{equation}
\lim_{k\rightarrow \infty }\mu _{k}(A)=\mu (A)  \label{convcjtos}
\end{equation}
for every set $A$ satisfying $\mu (\partial A)=0$.
\end{enumerate}
\end{lemma}
\begin{lemma}
\label{gooddk} There exist $\ k_{0}\in \mathbb{N}^{+}$ and two
sequences $ \{x_{k}\}$, $\{y_{k}\}\subset A_{k}$ such that
\begin{equation}
d_{k}:=dist(x_{k},y_{k})\in D_{k}\ \forall k\geq k_{0},
\label{newdk}
\end{equation}
and
\begin{equation}
m_{k}:=\frac{(2d_{k})^{s}}{\mu
_{k}(B^{\prime}(x_{k},d_{k}))}\rightarrow P^{s}(E)\quad\text{ as
}\quad k\rightarrow \infty  \label{mk}
\end{equation}
\end{lemma}
\begin{proof}
 Let $x_{0},y_{0}\in E$ be such that
\begin{equation}
P^{s}(E)=h(x_{0},d_{0})\text{ \ with \ }d_{0}:=dist(x_{0},y_{0})\in \lbrack
\tilde{c},\frac{\tilde{c}}{r_{\min }})  \label{maxpack}
\end{equation}
(see Corollaries \ref{maxallgo} and \ref{newattain}). By
Lemma~\ref{keylemaconv} (i) and Theorem~\ref{main}, we can take
two sequences $\{x_{k}\}$, $\{y_{k}\}$ with $x_{k}$, $y_{k}\in $
$A_{k}$ $\ \forall k\in \mathbb{N}^{+}$, such that
\begin{equation*}
x_{k}\rightarrow x_{0},\qquad y_{k}\rightarrow y_{0},\qquad
d_{k}:=dist(x_{k},y_{k})\rightarrow d_{0}
\end{equation*}
and
\begin{equation*}
m_{k}:=\frac{(2d_{k})^{s}}{\mu _{k}(B^{\prime
}(x_{k},d_{k}))}\rightarrow P^{s}(E)\quad \text{as } k\rightarrow
\infty .
\end{equation*}
Observe that $D_{k}$ is the set of those distances between points
of $A_{k}$ within the interval
$[\tilde{c},\frac{\tilde{c}}{r_{\min }}]$. We claim that, either
there exists $k_{0}\in \mathbb{N}^{+}$ such that $d_{k}\in D_{k} $
for every $k\geq k_{0}$, or we can construct two other sequences
$\{x_{k}^{\prime }\}$, $\{y_{k}^{\prime }\}$ with $x_{k}^{\prime
},y_{k}^{\prime }\in A_{k}\ \forall k\in \mathbb{N}^{+}$,
satisfying (\ref{newdk}) and (\ref{mk}). It is clear that if
$d_{0}\in (\tilde{c},\frac{\tilde{c}}{r_{\min }})$ we are in the
first case and the theorem holds. Suppose now that
$d_{0}=\tilde{c}$ and let $\mathbf{i=}i_{1}i_{2}...\in \mathbb{M}$
be such that $\pi (\mathbf{i})=x_{0}$. Then, as $\tilde{c}\leq c$
(see Remark~\ref{relations}), $B^{\prime }(x_{0},d_{0})\cap
E\subset E_{i_{1}}$. Therefore, we can apply
Lemma~\ref{bolatras}~(i) to get
\begin{equation}
h(x_{0},d_{0})=h(f_{i_{1}}^{-1}(x_{0}),\frac{d_{0}}{r_{i}})\text{,}
\label{eq}
\end{equation}
where $\frac{d_{0}}{r_{i_{1}}}%
=dist(f_{i_{1}}^{-1}(x_{0}),f_{i_{1}}^{-1}(y_{0}))$. Let
\begin{equation*}
\{d_{k}^{1}\}:=\{d_{k}:d_{k}\geq \tilde{c}\}\qquad \text{  and
}\qquad \{d_{k}^{2}\}:=\{d_{k}:d_{k}<\tilde{c}\}.
\end{equation*}%
and suppose without loss of generality that both subsequences are
not finite. For $j=1,2$, denote by $\{x_{k}^{j}\}$\ and
$\{y_{k}^{j}\}$\ the convergent subsequences of $\{x_{k}\}$\ and
$\{y_{k}\}$, respectively, such that
$\{d_{k}^{j}\}=\{dist(x_{k}^{j},y_{k}^{j})\}$. Now, by
construction there exists $k_{0}\in \mathbb{N}^{+}$\ such that
$\tilde{c}\leq d_{k}^{1}\leq \frac{\tilde{c}}{r_{\min }}$\ $
\forall k\geq k_{0}$\ and, by the SSC and (\ref{ckbound}), there
exists $ k_{1}\in \mathbb{N}^{+}$\ such that $x_{k}^{2},$
$y_{k}^{2}\in A_{k}\cap E_{i_{1}}$for all $ k\geq k_{1}$. Hence,
by Theorem~\ref{main} and (\ref{eq}),
\begin{equation*}
\frac{(2\frac{d_{k}^{2}}{r_{i_{1}}})^{s}}{\mu _{k}(B^{\prime
}(f_{i_{1}}^{-1}(x_{k}^{2}),\frac{d_{k}^{2}}{r_{i_{1}}}))}\rightarrow
h(f_{i_{1}}^{-1}(x_{0}),\frac{d_{0}}{r_{i_{1}}})=P^{s}(E)\quad
\text{ as } j\rightarrow \infty .
\end{equation*}
Since $\frac{\tilde{c}}{r_{\min }}>\frac{d_{k}^{2}}{r_{i_{1}}}
=dist(f_{i_{1}}^{-1}(x_{k}^{2}),f_{i_{1}}^{-1}(y_{k}^{2}))\rightarrow
\frac{d_{0}}{r_{_{i_{1}}}}\in (\tilde{c},\frac{\tilde{c}}{r_{\min
}}]$, there exists $k_{2}\geq k_{1}$\ such that
$\frac{d_{k}^{2}}{r_{i_{1}}}\in \lbrack
\tilde{c},\frac{\tilde{c}}{r_{\min }}]$ $\forall k\geq k_{2}$.
Moreover, by Lemma~\ref{keylemaconv}~(ii),
$f_{i_{1}}^{-1}(x_{k}^{2}),$\ $ f_{i_{1}}^{-1}(y_{k}^{2})\in
A_{k}$\ $\forall k\geq k_{2}$. This concludes the proof of the
theorem as the sequences
\begin{equation*}
x_{k}^{\prime }:=\left\{
\begin{array}{c}
x_{k}^{1}\text{ if }d_{k}\geq \tilde{c} \\
f_{i_{1}}^{-1}(x_{k}^{2})\text{ if }d_{k}<\tilde{c}
\end{array}
\right. \qquad \text{\ and}\qquad \text{ }y_{k}^{\prime }:=\left\{
\begin{array}{c}
y_{k}^{1}\text{ if }d_{k}\geq \tilde{c} \\
f_{i_{1}}^{-1}(y_{k}^{2})\text{ if }d_{k}<\tilde{c}
\end{array}
\right.
\end{equation*}
satisfy (\ref{newdk}) and (\ref{mk}).
\end{proof}

We state as a lemma the following result extracted from the proof
of \cite[Theorem 4.13]{[LLM2]} that we shall need to prove the
convergence of $\{\tilde{m}_{k}\}$.
\begin{lemma}
\label{convstep}For any sequence $\{(x_{k},d_{k})\}_{k=1}^{\infty
}$\ in $ A\times D$\
\begin{equation*}
\lim_{k\rightarrow \infty }|\mu _{k}(B^{\prime
}(x_{k},d_{k}))-\mu(B(x_{k},d_{k}))|=0.
\end{equation*}
\end{lemma}
\begin{proof}
Let $\{(x_{k},d_{k})\}_{k=1}^{\infty }$\ \ be a sequence in
$A\times D$. For any\textbf{\ }$k\in \mathbb{N}^{+}$ satisfying
that $r_{\max }^{k}<c$, define
\begin{eqnarray*}
G_{k} &:&=\{E_{\mathbf{i}_{k}}:\mathbf{i}_{k}\in M^{k}\text{ and } E_{\mathbf{i}_{k}}\subset B^{\prime }(x_{k},d_{k})\} \\
P_{k} &:&=\{E_{\mathbf{i}_{k}}:\mathbf{i}_{k}\in M^{k}\text{,
}E_{\mathbf{i} _{k}}\cap B^{\prime }(x_{k},d_{k})\neq \emptyset
\text{ and }E_{\mathbf{i}_{k}}\cap (E\setminus B^{\prime }(x_{k},d_{k}))\neq \emptyset \} \\
R_{k} &:&=\{E_{\mathbf{i}_{k}}\in P_{k}:\text{ }A_{k}\cap
E_{\mathbf{i} _{k}}\cap B^{\prime }(x_{k},d_{k})\neq \emptyset \}.
\end{eqnarray*}
Then,
\begin{eqnarray*}
R_{k} &\subset &P_{k} \\
E\cap B^{\prime }(x_{k},d_{k}) &=&G_{k}\cup (B^{\prime }(x_{k},d_{k})\cap
P_{k}) \\
\mu _{k}(B^{\prime }(x_{k},d_{k})) &=&\mu _{k}(G_{k})+\mu _{k}(R_{k}) \\
\mu (B(x_{k},d_{k})) &=&\mu (B^{\prime }(x_{k},d_{k}))=\mu (G_{k})+\mu
(B^{\prime }(x_{k},d_{k})\cap P_{k})\qquad \text{and} \\
P_{k} &\subset &B^{\prime }(x_{k},d_{k}+r_{\max }^{k})\setminus
B(x_{k},d_{k}-r_{\max }^{k})
\end{eqnarray*}
Moreover, by (\ref{keyinq}), $\mu (G_{k})=\mu _{k}(G_{k})$ and
$\mu (R_{k})=\mu _{k}(R_{k})$. This, together with the triangle
inequality gives
\begin{eqnarray*}
|\mu _{k}(B^{\prime }(x_{k},d_{k}))-\mu (B(x_{k},d_{k}))| &=&|\mu
_{k}(R_{k})-\mu (B^{\prime }(x_{k},d_{k})\cap P_{k})|\leq \\
&\leq &\mu (R_{k})-\mu (B^{\prime }(x_{k},d_{k})\cap R_{k})+\mu (\left(
B^{\prime }(x_{k},d_{k})\cap P_{k}\right) \setminus R_{k}))\leq \\
&\leq &\mu (R_{k})+\mu (\left( B(x_{k},d_{k})\cap P_{k}\right) \setminus
R_{k})\leq \mu (P_{k})\leq \\
&\leq &\mu (B(x_{k},d_{k}+r_{\max }^{k})\setminus
B(x_{k},d_{k}-r_{\max }^{k})).
\end{eqnarray*}
Thus, the lemma holds if $\lim_{k\rightarrow \infty }\mu
(B(x_{k},d_{k}+r_{\max }^{k})\setminus B(x_{k},d_{k}-r_{\max
}^{k}))=0$. The last is true by the compactness of $E\times
\lbrack \tilde{c},\frac{\tilde{c}}{r_{\min }}]$, the continuity of
$\mu $ and the fact that $\mu (\partial B(x,d))=0$ (see
\cite{[MT]}).
 \end{proof}

We obtain, as immediate consequence of Lemma~\ref{convstep} and
(\ref{condition}), the following result useful in the proof of
Theorem~\ref{conv}.
\begin{corollary}
\label{corolconv}Let $\{(\tilde{x}_{k},\tilde{d}_{k})\}_{k=1}^{\infty }$ be
such that $\tilde{m}_{k}=h_{k}(\tilde{x}_{k},\tilde{d}_{k})$ $\forall k\in
\mathbb{N}^{+}$. Then
\begin{enumerate}
\item[i)] $\lim_{k\rightarrow
\infty}|h_{k}(\tilde{x}_{k},\tilde{d}_{k})-h(\tilde{x}_{k},\tilde{d}_{k})|=0$
\item[ii)] If $\{\tilde{m}_{k_{j}}\}_{j=1}^{\infty }\ $\ is a
convergent subsequence of $\{\tilde{m}_{k}\}_{k=1}^{\infty }$,
then \ $ \lim_{j\rightarrow
\infty}\tilde{m}_{k_{j}}=\lim_{j\rightarrow
\infty}h(\tilde{x}_{k_{j}},\tilde{d}_{k_{j}}).$
\end{enumerate}
\end{corollary}
We are now ready to show our main theorem.

\begin{proof}[Proof of Theorem~\protect\ref{conv}]
 (\ref{lowerbound}) is immediate since, for any $k\in \mathbb{N}^{+}$ such that
  $\mu_{k}(B^{\prime }(\widetilde{x}_{k},\widetilde{d}_{k}))=\mu(B(\widetilde{x}_{k},\widetilde{d}_{k}))$, (\ref{maxalg}) implies
that
\begin{equation*}
P^{s}(E)=\max_{(x,d)\in E\times
[\tilde{c},\frac{\tilde{c}}{r_{\min }}]}h(x,d)\geq
h(\widetilde{x}_{k},\widetilde{d}_{k})=h_{k}(\widetilde{x}_{k},
\widetilde{d}_{k})=\tilde{m}_{k}.
\end{equation*}
We turn now to the convergence. Observe that, although the
sequence $\tilde{m_{k}}$ is not necessarily monotone, for any $k
\in \mathbb{N}^{+}$, the argument of $h_{k}$ ranges on the compact
set $A\times [ \tilde{c},\frac{\tilde{c}}{r_{\min }}]$ and hence
we can take a convergent subsequence $\{\tilde{m}_{k_{j}}\}_{j \in
\mathbb{N}^{+}}$. It is enough to show that\textbf{\
}$\tilde{m}:=\lim_{j\rightarrow \infty
}\tilde{m}_{k_{j}}=P^{s}(E)$. We assume the subsequence
$\{\tilde{m}_{k_{j}}\}_{j\in \mathbb{N}^{+}}$ to be the whole
sequence and write $\tilde{m} :=\lim_{k\rightarrow
\infty}\tilde{m}_{k}$. We first show that $\tilde{m}\geq
P^{s}(E)$. Let $\{m_{k}\}_{k \in
\mathbb{N}^{+}}:=\{h_{k}(x_{k},d_{k})\}_{k\in \mathbb{N}^{+}}$ and
$k_{0}\in \mathbb{N}^{+}$ as in Lemma~\ref{gooddk}, then
$\lim_{k\rightarrow \infty}m_{k}=P^{s}(E)$ and
\begin{equation*}
d_{k}\in D_{k}\qquad \forall k\geq k_{0}.
\end{equation*}
Hence, if we suppose on the contrary that $\tilde{m}<P^{s}(E)$, we
can find $ k_{1}\in \mathbb{N}^{+}$ such that, for any $k\geq
k_{1}$,
\begin{equation}
\tilde{m}_{k}<m_{k}  \label{basiccontrad}
\end{equation}
in contradiction with (\ref{seal}). With the aim of showing the
reverse inequality, assume that $\tilde{m}>P^{s}(E)$. Then, by
Corollary~\ref{corolconv}\textit{~}ii), there exists $k_{2}\in
\mathbb{N}^{+}$ such that for any $k\geq k_{2}$
\begin{equation*}
P^{s}(E)<h(\tilde{x}_{k},\tilde{d}_{k}).
\end{equation*}
Moreover, by (\ref{maxalg}), we can take $(x_{0},d_{0})\in E\times
\lbrack \tilde{c},\frac{\tilde{c}}{r_{\min }}]$ such that
\begin{equation*}
P^{s}(E)=\max_{(x,d)\in E\times \lbrack
\tilde{c},\frac{\tilde{c}}{r_{\min }}
]}h(x,d)=h(x_{0},d_{0})<h(\tilde{x}_{k},\tilde{d}_{k})
\end{equation*}
which is a contradiction since \ $(\tilde{x}_{k},\tilde{d}_{k})\in
E\times \lbrack \tilde{c},\frac{\tilde{c}}{r_{\min }}]$. This
concludes the proof of the theorem.
 \end{proof}

\section{\protect\bigskip Examples\label{examples}}
In this section we test the algorithm with several examples that allow us to
explain how, besides providing empirical evidence of the precise (or
approximate) value of $P^{s}(E),$ the algorithm suggests a candidate to
optimal ball. In many cases, this information can be used to prove
rigorously that the value suggested by the algorithm is in fact the true
value of $P^{s}(E)$. We test this with some examples showing first which is
the information provided by the algorithm and secondly how this info can be
used to prove theorems giving the precise value of $P^{s}(E)$.
\subsection{Testing efficiency: The packing measure of a class of Sierpinski
gaskets with SSC.}
The algorithm presented in this job is specially useful to find the exact
value of the packing measure when the examples under\textbf{\ }consideration
have an \textit{stable} \textit{behavior} in the sense that both, the
ball(s) selected by the algorithm and the values of $\tilde{m}_{k}$ (see (%
\ref{mktildedef})) are the same in consecutive generations. In
these cases the results obtained could be considered as an
empirical evidence that the inverse density of the selected
ball(s) is going to give the precise value of the packing measure.
Actually,\ this ball(s) can be used to show that the value of the
packing measure equals to the corresponding $\tilde{m}_{k}$. This
is precisely the method that we are going to follow to prove
Theorem~\ref{packsier}.\ We first present a table of results
obtained from applying the algorithm to some members of the class
of Sierpinski gaskets defined by (\ref{sierclas}). The stability
observed (see Table~\ref{siertable}) allow us the obtention of
candidates for optimal balls. In the proof of Theorem
\ref{packsier}\ we make use of these candidates, that are the
actual optimal balls, to show \ that $P^{s(r)}(S_{r})=\left(
2\frac{1-r}{r}\right) ^{s(r)}$\ when $r\in (0,\frac{1}{3}]$. The
next table shows the results obtained after applying the algorithm
to different members of the class $S_{r}$ when $r\in
(0,\frac{1}{3}]$. Keeping the notation given in Section
\ref{description}, $\ \{x_{1},x_{2},x_{3}\}$ stands for the set of
fixed points of the similarities $\{f_{1,}f_{2,}f_{3} \} $,
respectively. More precisely ,
$\{x_{1},x_{2},x_{3}\}=\{(0,0),(1,0),(\frac{1}{2},\frac{\sqrt{3}}{2})\}$.
\mathstrut
\begin{table}
\begin{tabular}[b]{|c|c|c|}
\hline $r$ & $s(r)$ & Algorithm output \\
\hline $\frac{1}{3}$ & $1$ & \multicolumn{1}{|l|}{
\begin{tabular}{|c|c|c|c|c|}
%\hline
& \textbf{$\tilde{m}_{k}$} &
\textbf{$[\tilde{x}_{k},\tilde{y}_{k}]$} &
\textbf{$\tilde{d}_{k}$} & \textbf{$B(\tilde{x}_{k},\tilde{d}_{k})$} \\
\hline
\begin{tabular}{l}
\vspace{0.25cm} \\
$k=0$ \\
\vspace{0.25cm}
\end{tabular}
& $\ \ \ 6$ \ \ \ \  &
\begin{tabular}{c}
$\lbrack x_{1},x_{2}]$ \\
$\lbrack x_{2},x_{1}]$ \\
$\lbrack x_{3},x_{2}]$
\end{tabular}
& $1$ &
\begin{tabular}{c}
$B(x_{1},1)$ \\
$B(x_{2},1)$ \\
$B(x_{3},1)$
\end{tabular}
\\ \hline
\begin{tabular}{l}
\vspace{0.25cm} \\
$1\leq k\leq 8$ \\
\vspace{0.25cm}
\end{tabular}
&
\begin{tabular}{l}
$\quad 4=\left( 2\frac{1-r}{r}\right) ^{s(r)} \ \quad$
\end{tabular}
&
\begin{tabular}{c}
$\lbrack x_{1},f_{2}(x_{0})]$ \\
$\lbrack x_{2},f_{1}(x_{2})]$ \\
$\lbrack x_{3},f_{1}(x_{3})]$
\end{tabular}
&$\frac{2}{3}$&
\begin{tabular}{c}
$B(x_{1},\frac{2}{3})$ \\
$B(x_{2},\frac{2}{3})$ \\
$B(x_{3},\frac{2}{3})$
\end{tabular}
\\ %\hline
\end{tabular}
} \\ \hline
$\frac{1}{4}$ & $\frac{\log 3}{\log 4}$ &
\multicolumn{1}{|l|}{
\begin{tabular}{|c|c|c|c|c|}
%%\hline
% & $\tilde{m}_{k}$ & $[\tilde{x}_{k},\tilde{y}_{k}]$ &
%$\tilde{d}_{k}$ & $B(\tilde{x}_{k},\tilde{d}_{k})$ \\ \hline
\begin{tabular}{l}
\vspace{0.25cm} \\
$k=0$ \\
\vspace{0.25cm}
\end{tabular}
& $5.1961$ &
\begin{tabular}{c}
$\lbrack x_{1},x_{2}]$ \smallskip\\
$\lbrack x_{2},x_{1}]$ \smallskip \\
$\lbrack x_{3},x_{2}]$ \smallskip
\end{tabular}
& $1$ \smallskip &
\begin{tabular}{c}
$B(x_{1},1)$ \\
$B(x_{2},1)$ \\
$B(x_{3},1)$
\end{tabular}
\\ \hline
\begin{tabular}{l}
\vspace{0.25cm} \\
$1\leq k\leq 7$ \\
\vspace{0.25cm}
\end{tabular}
&
\begin{tabular}{l}
$4.1368=\left( 2\frac{1-r}{r}\right) ^{s(r)}$
\end{tabular}
&
\begin{tabular}{c}
$\lbrack x_{1},f_{2}(x_{0})]$ \\
$\lbrack x_{2},f_{1}(x_{2})]$ \\
$\lbrack x_{3},f_{1}(x_{3})]$
\end{tabular}
& $\frac{3}{4}$ \smallskip &
\begin{tabular}{c}
$B(x_{1},\frac{3}{4})$ \smallskip\\
$B(x_{2},\frac{3}{4})$ \smallskip\\
$B(x_{3},\frac{3}{4})$ \smallskip
\end{tabular}
\\ %\hline
\end{tabular}
} \\ \hline $\frac{2}{10}$ & $\frac{-\log 3}{\log 0.2}$ &
\multicolumn{1}{|l|}{
\begin{tabular}{|c|c|c|c|c|}
%\hline & $\tilde{m}_{k}$ & $[\tilde{x}_{k},\tilde{y}_{k}]$ &
%$\tilde{d}_{k}$ & $B(\tilde{x}_{k},\tilde{d}_{k})$ \\ \hline
\begin{tabular}{l}
\vspace{0.25cm} \\
$k=0$ \\
\vspace{0.25cm}
\end{tabular}
& $4.815$ &
\begin{tabular}{c}
$\lbrack x_{1},x_{2}]$ \\
$\lbrack x_{2},x_{1}]$ \\
$\lbrack x_{3},x_{2}]$
\end{tabular}
& $1$ &
\begin{tabular}{c}
$B(x_{1},1)$ \\
$B(x_{2},1)$ \\
$B(x_{3},1)$
\end{tabular}
\\ \hline
\begin{tabular}{l}
\vspace{0.25cm} \\
$1\leq k\leq 7$ \\
\vspace{0.25cm}
\end{tabular}
&
\begin{tabular}{l}
$4.1348=\left( 2\frac{1-r}{r}\right) ^{s(r)}$
\end{tabular}
&
\begin{tabular}{c}
$\lbrack x_{1},f_{2}(x_{0})]\smallskip $ \\
$\lbrack x_{2},f_{1}(x_{2})]\smallskip$ \\
$\lbrack x_{3},f_{1}(x_{3})]\smallskip $
\end{tabular}
& $\frac{8}{10}$ &
\begin{tabular}{c}
$B(x_{1},\frac{8}{10})$ \\
$B(x_{2},\frac{8}{10})$ \\
$B(x_{3},\frac{8}{10})$
\end{tabular}
\\ %\hline
\end{tabular}
} \\ \hline
 $\frac{1}{27}$ & $\frac{\log 3}{\log 27}$ &
\multicolumn{1}{|l|}{
\begin{tabular}{|c|c|c|c|c|}
%\hline
%& $\tilde{m}_{k}$ & $[\tilde{x}_{k},\tilde{y}_{k}]$ &
%$\tilde{d}_{k}$ & $B(\tilde{x}_{k},\tilde{d}_{k})$ \\ \hline
\begin{tabular}{l}
\vspace{0.25cm} \\
$k=0$ \\
\vspace{0.25cm}
\end{tabular}
& $3.7798$ &
\begin{tabular}{c}
$\lbrack x_{1},x_{2}]$ \\
$\lbrack x_{2},x_{1}]$ \\
$\lbrack x_{3},x_{2}]$
\end{tabular}
& $1$ &
\begin{tabular}{c}
$B(x_{1},1)$ \\
$B(x_{2},1)$ \\
$B(x_{3},1)$
\end{tabular}
\\ \hline
\begin{tabular}{l}
\vspace{0.25cm} \\
$1\leq k\leq 7$ \\
\vspace{0.25cm}
\end{tabular}
&
\begin{tabular}{l}
$3.7325 =\left( 2\frac{1-r}{r}\right) ^{s(r)}$
\end{tabular}
&
\begin{tabular}{c}
$\lbrack x_{1},f_{2}(x_{0})]$ \\
$\lbrack x_{2},f_{1}(x_{2})]$ \\
$\lbrack x_{3},f_{1}(x_{3})]$
\end{tabular}
& $\frac{26}{27}$ &
\begin{tabular}{c}
$B(x_{1},\frac{26}{27})$ \\
$B(x_{2},\frac{26}{27})$ \\
$B(x_{3},\frac{26}{27})$
\end{tabular}
\\ %\hline
\end{tabular}
} \\ \hline
\end{tabular}
\caption{Algorithm output for some members of the class $S_r$. }
\label{siertable}
\end{table}

\begin{figure}[h]
\includegraphics[width=0.6\textwidth]{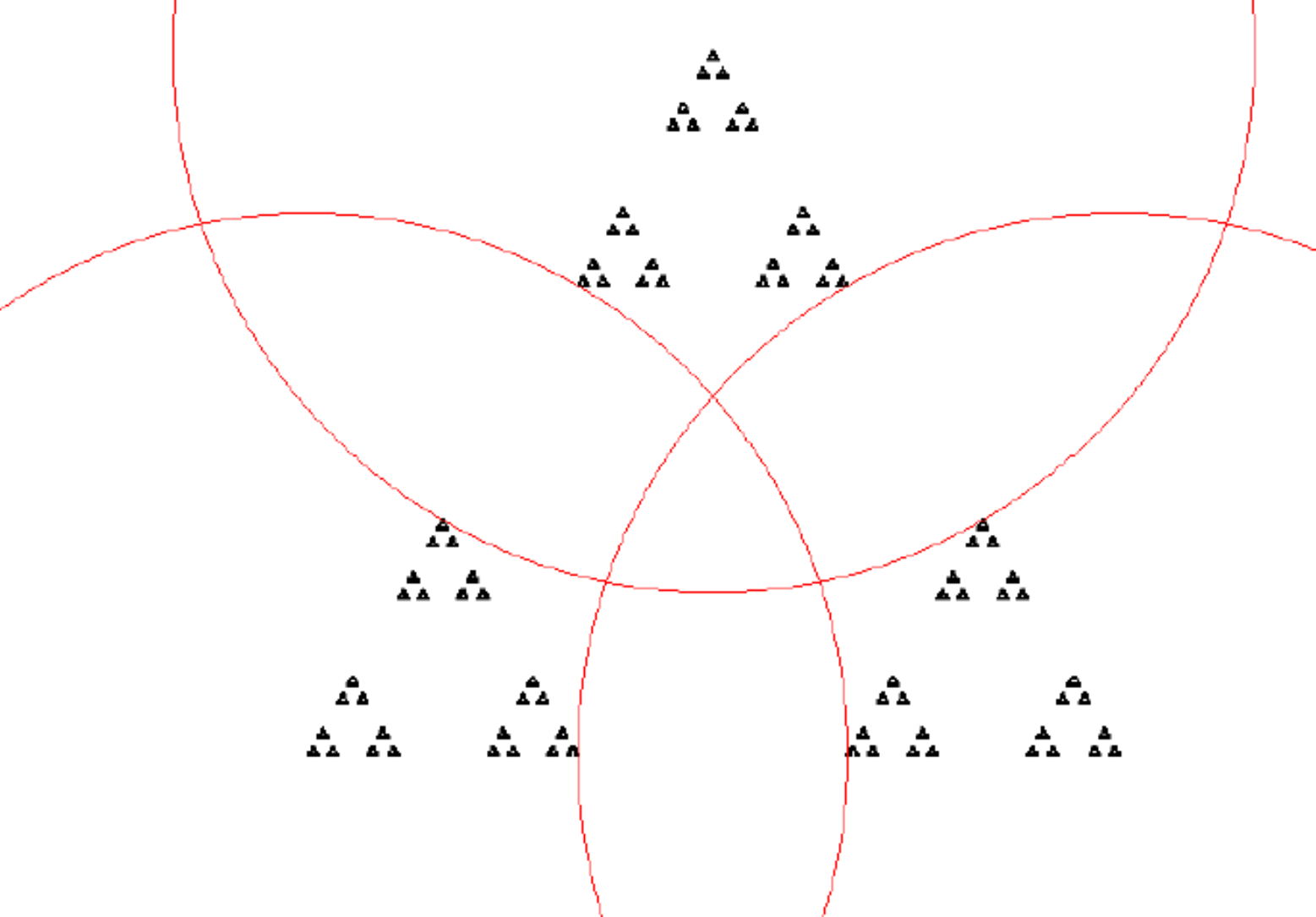}
\includegraphics[width=0.6\textwidth]{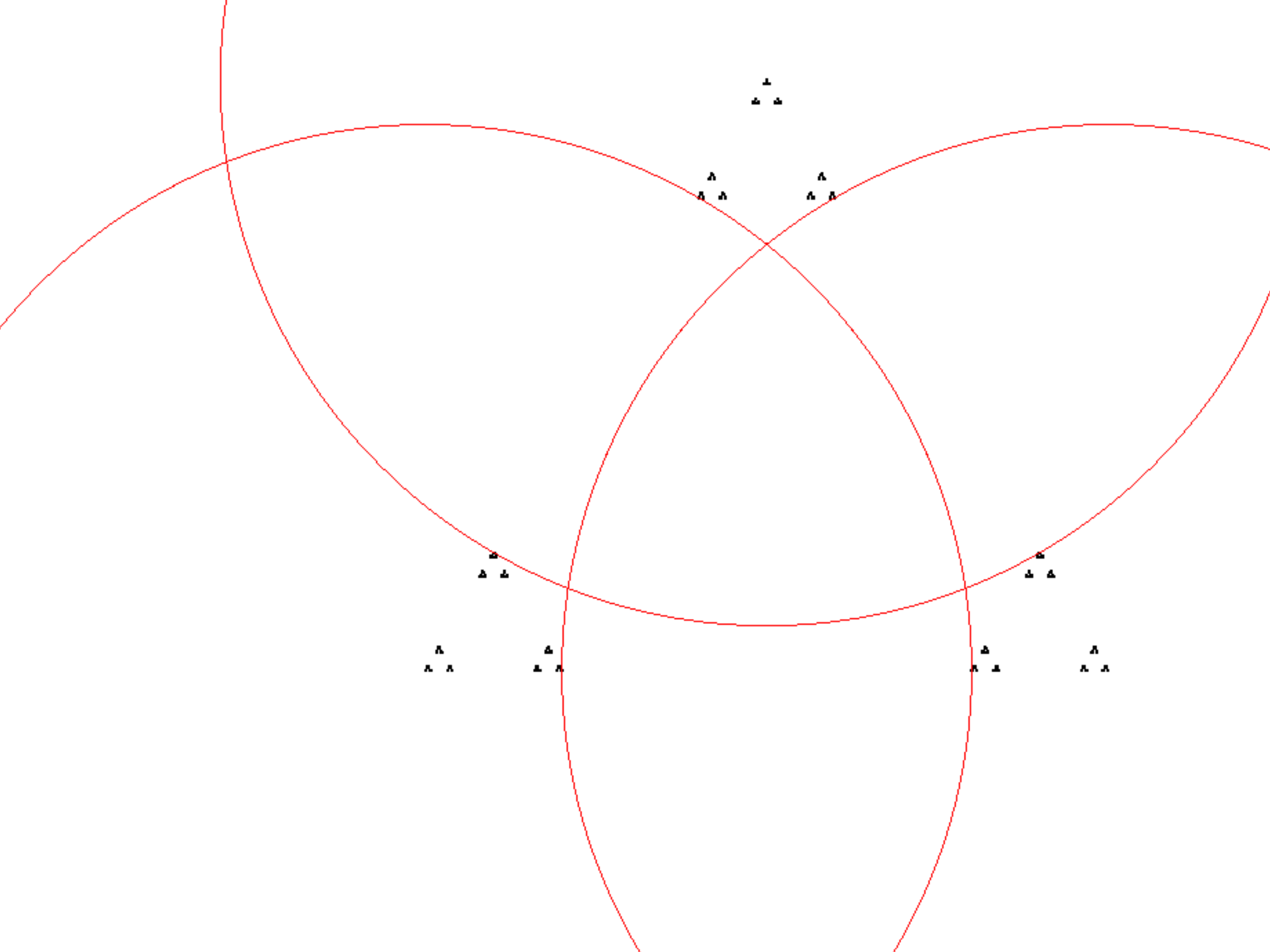}
%\centering
\caption{Algorithm selected balls for $S_{\frac{1}{3}}$ and
$S_{\frac{2}{10}}$.}\label{sier13}
\end{figure}

\mathstrut
From the results given in Table~\ref{siertable}, we can
conjecture that
\begin{equation*}
P^{s(r)}(S_{r})=h(x_{i},1-r)=\left(
2\frac{1-r}{r}\right)^{s(r)}\text{ where }i=1,2,3\text{;}
\end{equation*}
provided\textbf{\ }$r\in (0,\frac{1}{3}]$. This conjecture is
based on the fact that the algorithm selects in all the cases the
three same balls, namely $B(x_{1},1-r)$, $B(x_{2},1-r)$ and
$B(x_{3},1-r)$ (see Figure~\ref{sier13}), all having inverse
density equal to $\left( 2 \frac{1-r}{r}\right) ^{s}$. Therefore
one can say that there is empirical evidence that the inverse
density, $h(x_{i},1-r)$, of any of these balls is going to give
the precise value of $P^{s(r)}(S_{r})$ when $r\in
(0,\frac{1}{3}]$. We turn now to the proof of
Theorem~\ref{packsier} that states the above conjecture to be
true.

\begin{proof}[Proof of Theorem~\protect\ref{packsier}]
The lower bound holds trivially by taking $(x,d)=((0,0),1-r)$ in
(\ref{packssc3}). The upper bound follows from
Theorem~\ref{keylema} if the following inequality holds
\begin{equation}
h(x,d)\leq \left( 2\frac{1-r}{r}\right) ^{s(r)}\text{ }\forall
x\in S_{r} \text{ and }1-2r\leq d\leq \frac{1-2r}{r}.
\label{inequ}
\end{equation}
Due to the symmetry of the Sierpinski gaskets, we need to show (\ref{inequ})
only for $x\in f_{1}(S_{r})$ and $1-2r\leq d\leq \frac{1-2r}{r}$. Moreover,
by the geometry of the picture one can see that, for a fixed $d$ $\in
\lbrack 1-2r,\frac{1-2r}{r}]$ and every $x\in f_{1}(S_{r})$,
\begin{equation}
\mu (B(x,d))\geq \mu (B((0,0),d))\text{.}  \label{restrict}
\end{equation}%
This is because moving the center of the ball from $(0,0)$ in any
direction within the Sierpinski gasket, can only increase the
amount of set $S_{r}$ lying in the ball. More precisely, for any
$1-2r\leq d\leq \frac{1-2r}{r}$ and $x\in f_{1}(S_{r})$,
\begin{equation}
B((0,0),d)\cap S_{r}\subset B(x,d)\cap S_{r}\text{.}  \label{inclusion}
\end{equation}
We now proof (\ref{inclusion}). Observe that, given $y_{i}\in f_{i}(S_{r})$,
$i\in \{2,3\}$,
\begin{equation}
\max_{x\in (S_{r})_{1}\text{ }}dist(x,y_{i})=dist((0,0),y_{i})\text{.}
\label{max}
\end{equation}
Clearly (\ref{inclusion}) holds if, for any $x\in f_{1}(S_{r})$,
$B((0,0),d)\cap f_{i}(S_{r})\subset B(x,d)\cap S_{r}$ for $i\in
\{1,2,3\}$. It is enough to check this inclusion when $i=2,3$
since $1-2r>r$ and hence, $ B((0,0),d)\cap f_{1}(S_{r})\subset
f_{1}(S_{r})\subset B(x,d)\cap S_{r}$. Let $y\in B((0,0),d)\cap
f_{i}(S_{r})$, then (\ref{max}) implies that $ dist(x,y)\leq
dist((0,0),y)\leq d$ and hence $y\in $ $B((x,d)\cap f_{i}(S_{r})$.
This shows that $B((0,0),d)\cap f_{i}(S_{r})\subset B((x,d)\cap
f_{i}(S_{r})$ for $i\in \{2,3\}$ and concludes the proof of
(\ref{inclusion}), which in turn implies (\ref{restrict}). It
remains to prove that $h((0,0),d)\leq \left( 2\frac{1-r}{r}\right)
^{s(r)} $ for every $1-2r\leq d\leq \frac{1-2r}{r}$. By
Corollary~\ref{newattain}, we need to check the upper bound only
when $d\in \lbrack 1-2r,1] $. We divide the proof in the following
two cases:
\begin{enumerate}
\item If $d\in \left[ 1-2r,1-r\right] $ then, $\mu (B((0,0),d))=r^{s(r)}$
and hence%
\begin{equation*}
h((0,0),d)=\frac{(2d)^{s(r)}}{r^{s(r)}}\leq \left( 2\frac{1-r}{r}\right)
^{s(r)}\text{.}
\end{equation*}
\item If $d\in \lbrack 1-r,1]$, we write
\begin{equation*}
\lbrack 1-r,1]=\cup _{k=1}^{\infty }[1-r+r^{k+1},1-r+r^{k}]
\end{equation*}%
and show that, $h((0,0),d)\leq \left( 2\frac{1-r}{r}\right)
^{s(r)}$ for any $d\in \lbrack 1-r+r^{k+1},1-r+r^{k}]$ \ with
$k\geq 1$. Observe that, in this case, $\mu (B((0,0),d)\geq
r^{s(r)}+2r^{(k+1)s(r)}$ and thus,
\begin{equation*}
h((0,0),d)\leq \frac{2^{s(r)}(1-r+r^{k})^{s(r)}}{r^{s(r)}+2r^{(k+1)s(r)}}.
\end{equation*}
So we need to show
\begin{equation*}
\frac{(1-r+r^{k})^{s(r)}}{r^{s(r)}+2r^{(k+1)s(r)}}\leq
\frac{(1-r)^{s(r)}}{r^{s(r)}},
\end{equation*}
or, equivalently,
\begin{equation}
\frac{(1-r+r^{k})^{s(r)}}{(1-r)^{s(r)}}=(1+\frac{r^{k}}{1-r})^{s(r)}\leq
\frac{r^{s(r)}+2r^{(k+1)s(r)}}{r^{s(r)},}=1+2r^{ks(r)}.  \label{expresion}
\end{equation}
Let $g_{1}(t)=(1+\frac{t}{1-r})^{s(r)}$ and
$g_{2}(t)=1+2t^{s(r)}$, both functions taking the same value at
$t=0$. Then, (\ref{expresion}) holds if $g_{1}^{\prime }(t)\leq
g_{2}^{\prime }(t)$ $\forall t\in \lbrack 0,\frac{1}{3}]$ or,
equivalently,
\begin{equation*}
\frac{1}{1-r}(\frac{1}{t}+\frac{1}{1-r})^{s(r)-1}\leq 2.
\end{equation*}
The last inequality is true because $r,t\leq \frac{1}{3}$ and $s(r)-1\leq 0$
and, therefore
\begin{equation*}
\frac{1}{1-r}(\frac{1}{t}+\frac{1}{1-r})^{s(r)-1}\leq \frac{3}{2}
(3+1)^{s(r)-1}\leq 2.
\end{equation*}
\end{enumerate}
\end{proof}
\begin{remark}
What happens\ when $r>\frac{1}{3}$? In these cases we have
observed that, if the contractio factors are not bigger than
$0.365$\textbf{,} then the selected balls are still
$B(x_{i},1-r)$, $i=1,2,3$. Notice that showing (\ref{sierfor}) for
$r\in [ \frac{1}{3},0.365]$ requires a modification of the
arguments given in the proof of Theorem~\ref{packsier}. Namely, as
$s(r)\geq 1$, it will be necessary to find the right decomposition
of the interval $[1-r,1]$. Finally, we have also noticed a loss of
stability in the numerical results when $r\in (0.365,0.5)$.
Namely, the selected ball varies on a small scale from one
iteration to the next, that is, only an approximation of the
optimal ball is reached since the center of the selected ball
remains fixed  but the radius changes slightly (see tables
\ref{siertable37} and \ref{siertable42} for the cases $r=0.37$ and
$r=0.42$).
\begin{table}
\begin{tabular}{|cccc|}
\hline &
\multicolumn{1}{|c}{$\tilde{m}_{k}=\frac{(2\tilde{d}_{k})^{s}}{\mu_{k}(B(\tilde{x}_{k},\tilde{d}_{k}))}$}
& \multicolumn{1}{|c}{$[\tilde{x}_{k}, \tilde{y}_{k}]$} &
\multicolumn{1}{|c|}{$\tilde{d}_{k}=dist(\tilde{x}_{k},
\tilde{y}_{k})$} \\ \hline $k=1$ & \multicolumn{1}{|c}{$6.452802$}
& \multicolumn{1}{|c}{$ [x_{1},f_{1}(x_{3})]$} &
\multicolumn{1}{|c|}{$0.37$} \\ \hline $k=2$ &
\multicolumn{1}{|c}{$3.872817$} &
\multicolumn{1}{|c}{$[x_{1},f_{2}(x_{1})]$} &
\multicolumn{1}{|c|}{$0,63$} \\ \hline $k=3$ &
\multicolumn{1}{|c}{$3.872817$} &
\multicolumn{1}{|c}{$[x_{1},f_{2}(x_{1})]$} &
\multicolumn{1}{|c|}{$0,63$} \\ \hline $k=4$ &
\multicolumn{1}{|c}{$3.872817$} &
\multicolumn{1}{|c}{$[x_{1},f_{2}(x_{1})]$} &
\multicolumn{1}{|c|}{$0,63$} \\ \hline $k=5$ &
\multicolumn{1}{|c}{$3.872817$} &
\multicolumn{1}{|c}{$[x_{1},f_{2}(x_{1})]$} &
\multicolumn{1}{|c|}{$0,63$} \\ \hline $k=6$ &
\multicolumn{1}{|c}{$3.872817$} &
\multicolumn{1}{|c}{$[x_{1},f_{2}(x_{1})]$} &
\multicolumn{1}{|c|}{$0,63$} \\ \hline $k=7$ &
\multicolumn{1}{|c}{$3.872817$} &
\multicolumn{1}{|c}{$[x_{1},f_{2}(x_{1})]$} &
\multicolumn{1}{|c|}{$0,63$} \\ \hline $k=8$ &
\multicolumn{1}{|c}{$3.872830$} &
\multicolumn{1}{|c}{$[x_{1},f_{31111111}(x_{2})]$} &
\multicolumn{1}{|c|}{$0.630176$}
\\ \hline
$k=9$ & \multicolumn{1}{|c}{$3.872865$} &
\multicolumn{1}{|c}{$[x_{1},f_{311111111}(x_{2})]$} &
\multicolumn{1}{|c|}{$0.630065$}
\\ \hline
$k=10$ & \multicolumn{1}{|c}{$3.872849$} &
\multicolumn{1}{|c}{$[x_{1},f_{2111111111}(x_{3})]$} &
\multicolumn{1}{|c|}{$0.630024$}
\\ \hline
$k=11$ & \multicolumn{1}{|c}{$3.87283414017915$} &
\multicolumn{1}{|c}{$[x_{1},f_{31111111111}(x_{2})]$} &
\multicolumn{1}{|c|}{$0.630009$} \\ \hline
\end{tabular}
\caption{Algorithm  output  for  $S_{0.37},$ $s=\frac{-\log
3}{\log 0.37}$.} \label{siertable37}
\end{table}

\begin{table}
\begin{tabular}{|cccc|}
\hline &
\multicolumn{1}{|c|}{$\tilde{m}_{k}=\frac{(2\tilde{d}_{k})^{s}}{\mu_{k}(B(\tilde{x}_{k},\tilde{d}_{k}))}$}
& \multicolumn{1}{|c|}{$[\tilde{x}_{k}, \tilde{y}_{k}]$} &
\multicolumn{1}{|c|}{$\tilde{d}_{k}=dist(\tilde{x}_{k},
\tilde{y}_{k})$} \\ \hline $k=1$ &
\multicolumn{1}{|c|}{$2.125979$}
& \multicolumn{1}{|c|}{$[f_{1}(x_{2}),f_{2}(x_{1})]$} & \multicolumn{1}{|c|}{$0.16$} \\
\hline $k=2$ & \multicolumn{1}{|c}{$7.216871$} &
\multicolumn{1}{|c|}{$[x_{1},f_{11}(x_{2})]$} & \multicolumn{1}{|c|}{$(0.42)^{2}$} \\
\hline $k=3$ & \multicolumn{1}{|c}{$3.659900$} &
\multicolumn{1}{|c|}{$[x_{1},f_{131}(x_{2})]$} & \multicolumn{1}{|c|}{$0.287885$} \\
\hline $k=4$ & \multicolumn{1}{|c}{$3.670508$} &
\multicolumn{1}{|c}{$[x_{1},f_{1311}(x_{2})]$} & \multicolumn{1}{|c|}{$0.260556$} \\
\hline $k=5$ & \multicolumn{1}{|c}{$3.658307$} &
\multicolumn{1}{|c}{$[x_{1},f_{13111}(x_{2})]$} & \multicolumn{1}{|c|}{$0.250391$} \\
\hline $k=6$ & \multicolumn{1}{|c}{$3.642973$} &
\multicolumn{1}{|c}{$[x_{1},f_{121111}(x_{3})]$} & \multicolumn{1}{|c|}{$0.246390$} \\
\hline $k=7$ & \multicolumn{1}{|c}{$3.633895$} &
\multicolumn{1}{|c}{$[x_{1},f_{1211111}(x_{2})]$} & \multicolumn{1}{|c|}{$0.245905$} \\
\hline $k=8$ & \multicolumn{1}{|c}{$3.630715$} &
\multicolumn{1}{|c}{$[x_{1},f_{12111131}(x_{1})]$} &
\multicolumn{1}{|c|}{$0.245207$}
\\ \hline $k=9$ & \multicolumn{1}{|c}{$3.629988$} &
\multicolumn{1}{|c}{$[x_{1},f_{121111321}(x_{3})]$} &
\multicolumn{1}{|c|}{$0.246752$}
\\ \hline $k=10$ & \multicolumn{1}{|c}{$3.629498$} &
\multicolumn{1}{|c}{$[x_{1},f_{1211113323}(x_{2})]$} &
\multicolumn{1}{|c|}{$0.246745$}
\\ \hline $k=11$ & \multicolumn{1}{|c}{$3.629288$} &
\multicolumn{1}{|c}{$[x_{1},f_{13111122322}(x_{2})]$} &
\multicolumn{1}{|c|}{$0.246662$} \\ \hline
\end{tabular}
\caption{Algorithm  output  for $S_{0.42},$  $s=\frac{-\log
3}{\log 0.42}$.} \label{siertable42}
\end{table}
\end{remark}
\subsection{Further examples}
\begin{enumerate}
\item \textbf{Cantor sets in the real line} Let $C_{r}$ be the
linear Cantor set obtained as the attractor of the iterated
function system $\{f_{1}(x)=rx,$ $f_{2}(x)=1-r+rx\}$, $x\in [
0,1]$ and $0<r<\frac{1}{2}$. The numerical results arising from
the application of the algorithm to several Cantor sets belonging
to the class $C_{r}$ indicate that the optimal balls are either
$B(0,1-r)$ and $B(1,1-r)$, or images of these two balls with the
same density. More precisely, the pattern observed \ is that when
the contractio factor is smaller than
$\frac{3}{2}-\frac{1}{2}\sqrt{5}$, then $d=1-r\leq \frac{c}{r}$
and, since $B(0,1-r)$ and $B(1,1-r)$ are admissible balls (see
(\ref{packssc3})), these are the chosen ones. The situation
changes when $r<\frac{3}{2}-\frac{1}{2}\sqrt{5},$ since then $1-r$
is too big to be radius of an admissible ball. In these cases, $
B(0,1-r) $ and $B(1,1-r)$ are replaced by
$f_{\mathbf{i}_{1}(n)}B(0,1-r))$ and
$f_{\mathbf{i}_{2}(n)}(B(1,1-r))$, where, for $j=1,2$,
$\mathbf{i}_{j}(n)$ is the word formed by the letter $j$ repeated
$n$ times and $n=n(r)$ . Therefore, the numerical results indicate
that the right formula for the packing measure should be
\begin{eqnarray}\label{cantconj}
P^{s(r)}(C_{r}) &=&\frac{|B(x_{j},1-r)|^{s(r)}}{\mu
(B(x_{j},1-r))}=\frac{|f_{\mathbf{i}_{j}(n)}(B(x_{j},1-r))|^{s(r)}}{\mu (f_{\mathbf{i}_{j}(n)}(B(x_{j},1-r)))}=   \\
&=&\left( 2\frac{1-r}{r}\right) ^{s(r)},\text{ \ \ \ }j=1,2  \notag
\end{eqnarray}
where $s(r)=\frac{-\log 2}{\log r}$ is the similarity dimension of
$C_{r}$ and $\{x_{1},x_{2}\}=\{0,1\}$. A previous work by Feng
shows that (\ref{cantconj}) actually holds. In \cite{[Fen0]} the
author obtains by other means a general formula for the packing
measure of linear Cantor sets. Notice that (\ref{cantconj})
coincides with the formula given in the Sierpinski gasket case
(see Theorem \ref{packsier}).
\begin{remark}
After testing several examples we have noticed that the number of
steps needed to observe an stable behavior varies from $1$ to $4$.
If the contractio factors are near $0.5$, then it is clear that we
cannot expect stability from early\textbf{\ }iterations. This is
due to the big size of the contractio ratios that makes
$\tilde{m}_{k}$ $=h_{k}(\tilde{x}_{k}, \tilde{d}_{k})$ (see
\ref{seal}) to be a bad approximation of
$h(\tilde{x}_{k},\tilde{d}_{k})$ at early stages. However, even in
the worse cases the selected interval is the same from iteration
$4th$ to $15th$. Therefore, we can conclude that, empirically, the
algorithm is recovering the formula (\ref{cantconj}) given by Feng
in \cite{[Fen0]}.
\end{remark}
\item \textbf{Cantor sets in the plane } Let $K_{r}$ be the
attractor of the iterated function system $\Psi
=\{f_{1},f_{2},f_{3},f_{4}\}$ where
$f_{i}(\mathbf{x})=r\mathbf{x}+b_{i}$, $ i=1,2,3,4$,
$\mathbf{x}=(x,y)\in \mathbb{R}^{2}$, $0<r<\frac{1}{2}$,
$b_{1}=(0,0)$, $b_{2}=(1-r,0)$, $b_{3}=(1-r,1-r)$, and
$b_{4}=(0,1-r)$. Let $ \{x_{1},x_{2},x_{3},x_{4}\}$ be the set of
fix points of the similarities in the system $\Psi $, i.e.,
$\{x_{1},x_{2},x_{3},x_{4}\}= \{(0,0),(1,0),(1,1),(0,1)\}$. The
implementation of the algorithm to the family $K_{r}$\ leads to
results quite similar to those observed in the Sierpinski gasket
case: Whenever $r\leq \frac{1}{4}$, \ the algorithm selects the
same four balls, namely $B(x_{i},1-r)$ where $i=1,...,4$. This
means that the experimental results indicate that the right
formula for the packing measure of $K_{r}$ should be
\begin{equation}
P^{s(r)}(K_{r})=\frac{|B(x_{i},1-r)|^{s(r)}}{\mu
(B(x_{i},1-r))}=\left( 2 \frac{1-r}{r}\right) ^{s(r)},\quad
i=1,..,4  \label{cantpconj}
\end{equation}
$\forall r\in (0,\frac{1}{4}]$. Observe that for both, the planar
Cantor sets with dimension less than one and the linear Cantor
sets, one can argue as in the proof of Theorem~\ref{packsier} to
show that (\ref{cantconj}) and (\ref{cantpconj}) are the
corresponding actual formulas for the packing measure of these two
families ( $K_{r}$ $\ $with $0<r\leq \frac{1}{4}$ and $C_{r}$ with
$0<r<\frac{1}{2}$). For example, in the case of $K_{r}$ $\ $with
$0<r\leq \frac{1}{4}$ we still have that, if $d\in \lbrack
1-r+r^{k+1},1-r+r^{k}]$, then $\mu (B((0,0),d)\geq
r^{s(r)}+2r^{(k+1)s(r)}$ and therefore (\ref{expresion}) should
also hold. Notice that we have proved (\ref{expresion}) for every
$r,t\leq \frac{1}{3}$ and $s(r)-1\leq 0$, so it remains true when
$0<r\leq \frac{1}{4}$. In the case of planar Cantor sets having
dimension bigger than one, our experimental results indicate that
there is still a range of contractio factors for which
(\ref{cantpconj}) still holds. Namely, $B(x_{i},1-r)$ is selected
whenever we take $r\leq 0.35$ (see Figure \ref{cantors} for the
case $r=\frac{1}{3}$). Actually, in \cite{[BZZL]} and \cite{[BZ]}
 we can find a proof of (\ref{cantpconj}) for the cases
$r=\frac{1}{3}$ and $r\in (\frac{1}{4},\frac{\sqrt{2}}{4})$,
respectively. In both papers the authors used the classical
relation between the packing measure and the upper densities (see,
for example, \cite{[T]}) to obtain the corresponding formulas.
Finally, we have noticed that above these values of $r$ the
selected ball varies on a small scale from one iteration to the
next, meaning that only an approximation of the optimal ball is
reached. In these cases the observed selected balls are of the
form $f_{\mathbf{i}_{j}(n)}B(x_{j},1-r+\epsilon (r))$, where $j\in
\{1,2,3,4\}$, $ \mathbf{i}_{j}(n)$ is the word formed by the
letter $j$ repeated $n$ times, $n=n(r)$ and $\epsilon (r)$ is a
small number depending on $r$ (see Figure~\ref{cantors} for the
case $r=0.4$).

\begin{figure}[H]
\includegraphics[width=0.6\textwidth]{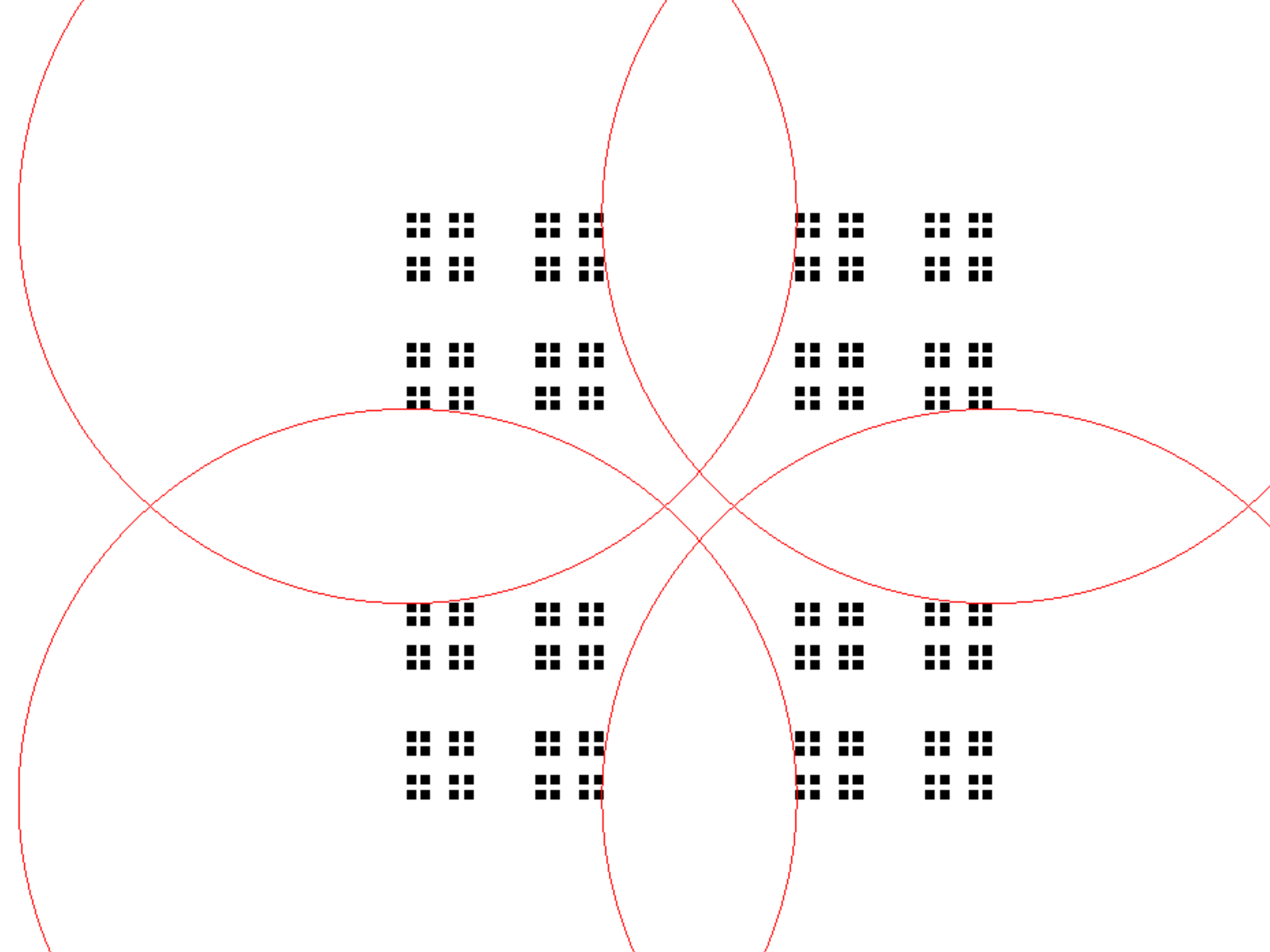}
\includegraphics[width=0.6\textwidth]{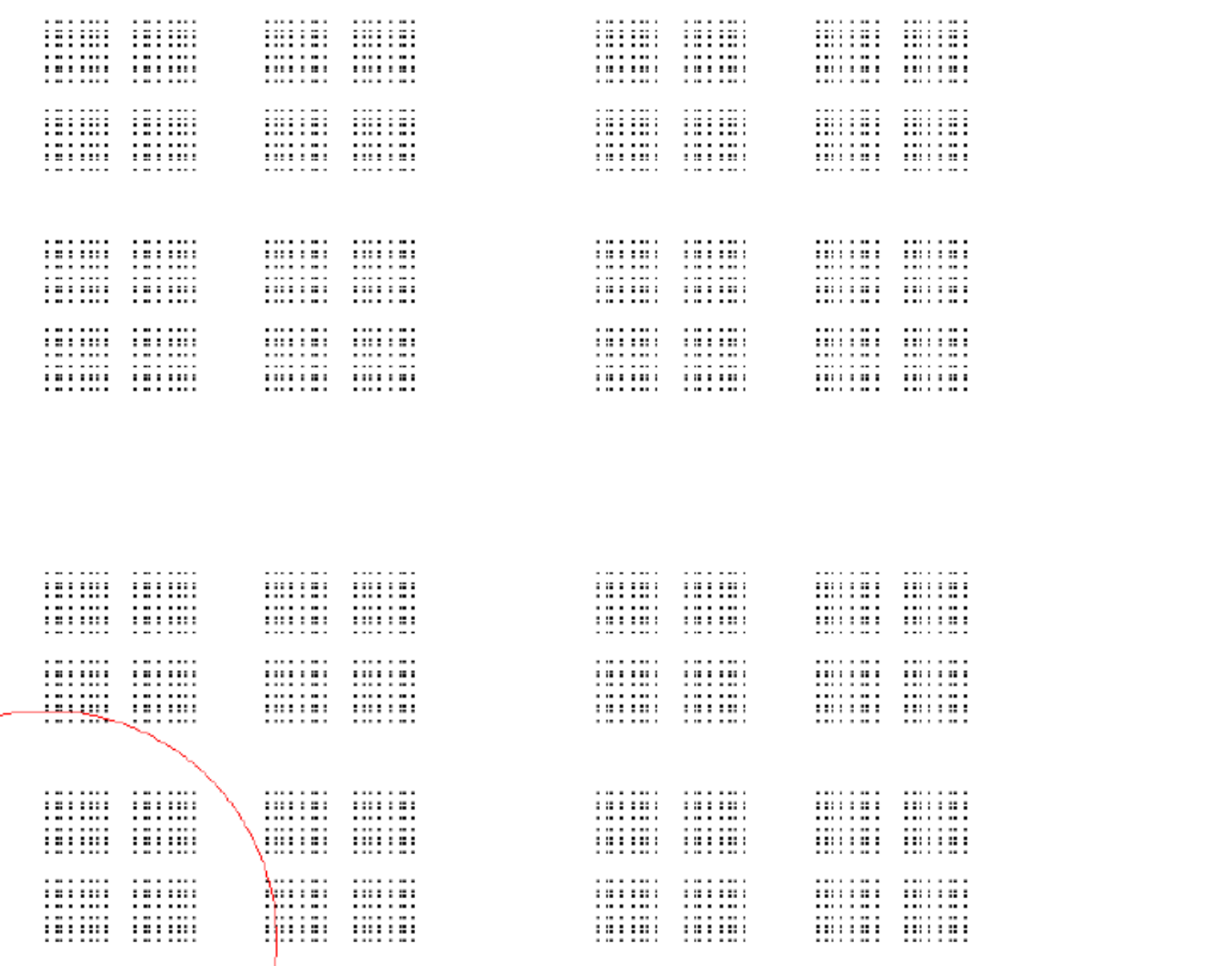}
%\centering
\caption{Algorithm selected balls for $K_{\frac{1}{3}}$ and
$K_{\frac{4}{10}}$.}
\label{cantors}
\end{figure}

The above results allows the following conjecture
\end{enumerate}
\begin{conjecture}
Let $E_{r}$ be the self-similar set associated to a system of
contracting similitudes $\Psi =\{f_{1,}f_{2,...,}f_{N}\}$
satisfying the SSC and such that ,$\forall i=1,...,N$,
$f_{i}(x)=rx+b_{i}$ , where $r\in (0,\frac{1}{N})$ , $x,b_{i}\in
\mathbb{R}^{n}$. Suppose that the fixed points of the $N$
similarities of the system are the vertices of an $N$-regular
polygon then
\begin{equation*}
P^{s(r)}(P_{r})=\frac{|B(x_{i},1-r)|^{s(r)}}{\mu
(B(x_{i},1-r))}=\left( 2 \frac{1-r}{r}\right)^{s(r)}
\end{equation*}
\end{conjecture}
\subsection*{Acknowledgement} We want to thank Maria Eugenia Mera
for her valuable help with the programming of the algorithm.

\end{document}